\documentclass[11pt]{amsart}

\usepackage{amsmath, amsthm, amssymb, amsfonts, enumerate}
\usepackage[colorlinks=true,linkcolor=blue,urlcolor=blue]{hyperref}
\usepackage{dsfont}
\usepackage{color}
\usepackage{geometry}
\usepackage{todonotes}

\newtheorem{theorem}{Theorem}[section]
\newtheorem{remark}[theorem]{Remark}

\newtheorem{lemma}[theorem]{Lemma}
\newtheorem{proposition}[theorem]{Proposition}

\title[]{Reflected BSDEs driven by $G$-Brownian motion with non-Lipschitz coefficients}
\author[Li]{Hanwu Li}
\address{H.\ Li: Research Center for Mathematics and Interdisciplinary Sciences, Shandong University, Qingdao 266237, Shandong, China.}
\address{H.\ Li: Frontiers Science Center for Nonlinear Expectations (Ministry of Education), Shandong University, Qingdao 266237, Shandong, China.}
\email{\href{mailto:lihanwu@sdu.edu.cn}{lihanwu@sdu.edu.cn}}
\date{\today}

\numberwithin{equation}{section}

\begin{document}

\begin{abstract}
In this paper, we consider the reflected backward stochastic differential equations driven by $G$-Brownian motion (reflected $G$-BSDEs) whose coefficients satisfy the $\beta$-order Mao's condition. The uniqueness is obtained by some a priori estimates and the existence can be proved by two different methods. The first one is Picard iteration and the second one is approximation via penalization. The latter construction is useful to get the comparison theorem.
\end{abstract}

\maketitle

\smallskip

{\textbf{Keywords}}:  $G$-BSDE, reflected $G$-BSDE, $\beta$-order Mao's condition, comparison theorem
\smallskip

{\textbf{MSC2010 subject classification}}: 60H10

%\smallskip

%{\textbf{JEL classification}}: D81, C61, G11.

\section{Introduction}
The nonlinear expectation theory, especially the $G$-expectation introduced by Peng, has attracted a great deal of attention recently. The motivation for studying $G$-expectation is to investigate financial problems in the market with volatility uncertainty and to provide a probabilistic interpretation for fully nonlinear partial differential equations. Peng systematically established a new kind of stationary and independent increment continuous process, called the $G$-Brownian motion, as well as the associated It\^{o}'s calculus. It is worth pointing out that the quadratic variation process of a $G$-Brownian motion is no longer deterministic. Based on this framework, Hu et al. \cite{HJPS1} investigated the stochastic differential equations driven by $G$-Brownian motion ($G$-BSDEs). The solution of a $G$-BSDE with terminal value $\xi$, generators $f,g$ is a triple of processes $(Y,Z,K)$ satisfying the following equation
\begin{displaymath}
Y_t=\xi+\int_t^T f(s,Y_s,Z_s)ds+\int_t^T g(s,Y_s,Z_s)d\langle B\rangle_s-\int_t^T Z_s dB_s-(K_T-K_t),
\end{displaymath} 
where the generators $f,g$ are Lipschitz continuous with respect to $(y,z)$. Compared with the classical BSDE theory, there is an additional nonincreasing $G$-martingale $K$ in this equation, which makes it difficult to construct a contraction mapping. The solution is constructed by applying the Galerkin approximation technique and the PDE approach. In the accompanying paper \cite{HJPS2}, the comparison theorem, Girsanov transformation and the nonlinear Feynman-Kac formula has been established. 

Recently, many efforts have been made to relax the Lipschitz assumption on the generators $f,g$. For instance, Hu, Lin and Soumana Hima \cite{HLS}  first investigated the $G$-BSDE with generators which have quadratic growth in the variable $z$. Hu, Qu and Wang \cite{HQW} introduced the $G$-BSDEs with time-varying coefficients. Under some appropriate assumptions, the terminal time can be extended to infinity. Wang and Zheng \cite{WZ} and Sun \cite{S} studied the case when the generators are uniformly continuous. More specifically, in \cite{WZ} the generators are Lipschitz continuous in the variable $y$ and uniformly continuous in $z$ while in \cite{S} the generators are uniformly continuous in $(y,z)$. He \cite{He} considered the $G$-BSDEs whose generators satisfy the so-called $\beta$-order Mao's condition and obtained the wellposedness results.

In some typical cases, such as pricing for American options, we need to study the case that the first component $Y$ of the solution to the $G$-BSDE is required to be above a given continuous process $S$, which leads to the study of reflected $G$-BSDEs. The prescribed process $S$ is called the lower obstacle and the solution to the reflected $G$-BSDE with parameters $(\xi,f,g,S)$ is a triple of processes $(Y,Z,A)$ satisfying
\begin{displaymath}
\begin{cases}
Y_t=\xi+\int_t^T f(s,Y_s,Z_s)ds+\int_t^T g(s,Y_s,Z_s)d\langle B\rangle_s
-\int_t^T Z_s dB_s+(A_T-A_t);\\
Y_t\geq S_t, 0\leq t\leq T, \textrm{ and } \{-\int_0^t (Y_s-S_s)dA_s\}_{t\in[0,T]} \textrm{ is a nonincreasing $G$-martingale.}
\end{cases}
\end{displaymath}
It is worth pointing out that due to the appearance of the nonincreasing $G$-martingale in $G$-BSDE, the nondecreasing process $A$ in the reflected $G$-BSDE satisfies the above so-called martingale condition instead of the Skorohod condition as in the classical case. Under the Lipschitz condition on the coefficients $f,g$ and assuming that the lower obstacle is either a $G$-It\^{o} process or a continuous process bounded from above, \cite{LPSH} obtained the existence and uniqueness result, the comparison theorem and the nonlinear Feynman-Kac formula for reflected $G$-BSDEs. Under the $G$-expectation framework, the upper obstacle case is significantly different from the lower obstacle case. Let $S$ be the upper obstacle, which means that the solution $Y$ should be smaller than this process. The solution of reflected $G$-BSDE with an upper obstacle $S$ is a triple $(Y,Z,A)$ satisfying the above conditions with $Y\geq S$ replacing by $Y\leq S$. Moreover, here $A$ is no longer a nondecreasing process but a bounded variation process, which makes it difficulty to derive a priori estimates. Li and Peng \cite{LP} applied the penalization method to construct the solution when the obstacle $S$ is a $G$-It\^{o} process and then showed that this solution is the maximal one using a variant comparison theorem. Recently, Li and Song \cite{LS} investigated the reflected $G$-BSDEs with double obstacles. By introducing the approximate Skorohod condition, they obtained the wellposedness of this problem when the upper obstacle is a generalized $G$-It\^{o} process, which makes it possible to relax the assumption on the obstacle processes for the single reflected case.

Note that the above mentioned reflected $G$-BSDEs require the Lipschitz assumption on the coefficients. To our best knowledge, \cite{CT} is the first paper to consider the reflected $G$-BSDEs with non-Lipschitz coefficients. More precisely, in \cite{CT} the coefficients $f,g$ have quadratic growth in the second variable $z$ and are Lipschitz continuous in $y$. Based on the non-reflected case \cite{HLS}, using the results of $G$-BMO martingale, $G$-Girsanov transformation, some a priori estimates are obtained and the solution  is constructed by approximation via penalization. In the present paper, we aim to study the reflected $G$-BSDEs when the coefficients does not satisfy the Lipschitz condition in the first variable $y$. Actually, we assume that $f,g$ satisfy the $\beta$-order Mao's condition (see (H1) in Section 2).

The paper is organized as follows. Section 2 is dedicated to preliminaries on the $G$-framework, the reflected $G$-BSDEs with Lipschitz coefficients and $G$-BSDEs with non-Lipschitz coefficients. In Section 3, we establish the uniqueness of reflected $G$-BSDEs under $\beta$-order Mao's condition. Then, we apply the Picard iteration method to obtain the existence result in Section 4. In Section 5, the solution is constructed by approximation via penalization, which will be helpful to derive the comparison theorem.

\section{Preliminaries}

In this section, we review some notations and results in the $G$-expectation framework, including the $G$-It\^{o} calculus,  the reflected $G$-BSDEs and the $G$-BSDEs under $\beta$-order Mao's condition. For simplicity, we only consider the one-dimensional case. For more details about $G$-expectation theory, we refer to the papers  %\cite{He}, \cite{LPSH}, \cite{LS}, % \cite{HJPS1}, \cite{HJPS2},
	\cite{P07a}, \cite{P08a}, \cite{P10}.

\subsection{$G$-expectation and $G$-It\^{o}'s calculus}

Let $\Omega=C_{0}([0,\infty);\mathbb{R})$, the space of
	real-valued continuous functions starting from the origin, be endowed
	with the following norm,
	\begin{displaymath}
		\rho(\omega^1,\omega^2):=\sum_{i=1}^\infty 2^{-i}[(\max_{t\in[0,i]}|\omega_t^1-\omega_t^2|)\wedge 1], \textrm{ for } \omega^1,\omega^2\in\Omega.
	\end{displaymath}
	Let  $B$ be the canonical
	process on $\Omega$. We define
	\[
	L_{ip} (\Omega):=\{ \varphi(B_{t_{1}},...,B_{t_{n}}):  \ n\in\mathbb {N}, \ t_{1}
	,\cdots, t_{n}\in\lbrack0,\infty), \ \varphi\in C_{b,Lip}(\mathbb{R}^{ n})\},
	\]
	where $C_{b,Lip}(\mathbb{R}^{ n})$ is the set of bounded Lipschitz functions on $\mathbb{R}^{n}$.  Let $(\Omega,L_{ip}(\Omega),\hat{\mathbb{E}})$ be the $G$-expectation space, where the function $G:\mathbb{R}\rightarrow\mathbb{R}$ is defined by
	\begin{displaymath}
		G(a):=\frac{1}{2}\hat{\mathbb{E}}[aB_1^2]=\frac{1}{2}(\bar{\sigma}^2a^+-\underline{\sigma}^2a^-).
	\end{displaymath}
 In this paper, we assume that $G$ is non-degenerate, i.e., $\underline{\sigma}^2 >0$. The (conditional) $G$-expectation for $\xi\in L_{ip}(\Omega)$ can be calculated as follows. Suppose that $\xi$ can be represented as
    \begin{displaymath}
    	\xi=\varphi(B_{{t_1}}, B_{t_2},\cdots,B_{t_n}).
\end{displaymath}
    Then, for $t\in[t_{k-1},t_k)$, $k=1,\cdots,n$,
\begin{displaymath}
	\hat{\mathbb{E}}_{t}[\varphi(B_{{t_1}}, B_{t_2},\cdots,B_{t_n})]=u_k(t, B_t;B_{t_1},\cdots,B_{t_{k-1}}),
\end{displaymath}
where, for any $k=1,\cdots,n$, $u_k(t,x;x_1,\cdots,x_{k-1})$ is a function of $(t,x)$ parameterized by $(x_1,\cdots,x_{k-1})$ such that it solves the following fully nonlinear PDE %defined on $[t_{k-1},t_k)\times\mathbb{R}$:
\begin{displaymath}
\begin{cases}
	\partial_t u_k+G(\partial_x^2 u_k)=0, (t,x)\in [t_{k-1},t_k)\times\mathbb{R};\\
	u_k(t_k,x;x_1,\cdots,x_{k-1})=u_{k+1}(t_k,x;x_1,\cdots,x_{k-1},x), k<n
\end{cases}
\end{displaymath}
and $u_n(t_n,x;x_1,\cdots,x_{n-1})=\varphi(x_1,\cdots,x_{n-1},x)$. %Hence, the $G$-expectation of $\xi$ is $\hat{\mathbb{E}}_0[\xi]$.
	
	For each $p\geq1$,   the completion of $L_{ip} (\Omega)$ under the norm $\Vert\xi\Vert_{L_{G}^{p}}:=(\hat{\mathbb{E}}[|\xi|^{p}])^{1/p}$ is denoted by $L_{G}^{p}(\Omega)$.  The conditional $G$-expectation $\mathbb{\hat{E}}_{t}[\cdot]$ can be
	extended continuously to the completion $L_{G}^{p}(\Omega)$. The canonical process $B$ is the 1-dimensional $G$-Brownian motion in this space.
		
	For each fixed $T\geq 0$, set $\Omega_T=\{\omega_{\cdot\wedge T}:\omega\in \Omega\}$. We may define $L_{ip}(\Omega_T)$ and $L_G^p(\Omega_T)$ similarly.  Besides, Denis, Hu and Peng \cite{DHP11} proved that the $G$-expectation has the following representation.
	\begin{theorem}[\cite{DHP11}]
		\label{the1.1}  There exists a weakly compact set
		$\mathcal{P}$ of probability
		measures on $(\Omega,\mathcal{B}(\Omega))$, such that
		\[
		\hat{\mathbb{E}}[\xi]=\sup_{P\in\mathcal{P}}E_{P}[\xi] \text{ for all } \xi\in  {L}_{G}^{1}{(\Omega)}.
		\]
		$\mathcal{P}$ is called the set that represents $\hat{\mathbb{E}}$.
	\end{theorem}
	
	Let $\mathcal{P}$ be a weakly compact set that represents $\hat{\mathbb{E}}$.
	For this $\mathcal{P}$, we define the capacity%
	\[
	c(A):=\sup_{P\in\mathcal{P}}P(A),\ A\in\mathcal{B}(\Omega).
	\]
	A set $A\in\mathcal{B}(\Omega_T)$ is called polar if $c(A)=0$.  A
		property holds $``quasi$-$surely"$ (q.s.) if it holds outside a
		polar set. In the following, we do not distinguish the two random variables $X$ and $Y$ if $X=Y$, q.s.
	
	For $\xi\in L_{ip}(\Omega_T)$, let $\mathcal{E}(\xi)=\hat{\mathbb{E}}[\sup_{t\in[0,T]}\hat{\mathbb{E}}_t[\xi]]$ and  $\mathcal{E}$ is called the $G$-evaluation. For $p\geq 1$ and $\xi\in L_{ip}(\Omega_T)$, define $\|\xi\|_{p,\mathcal{E}}=[\mathcal{E}(|\xi|^p)]^{1/p}$ and denote by $L_{\mathcal{E}}^p(\Omega_T)$ the completion of $L_{ip}(\Omega_T)$ under $\|\cdot\|_{p,\mathcal{E}}$. The following theorem can be regarded as  Doob's maximal inequality under $G$-expectation.
	\begin{theorem}[\cite{S11}]\label{the1.2}
		For any $\alpha\geq 1$ and $\delta>0$, $L_G^{\alpha+\delta}(\Omega_T)\subset L_{\mathcal{E}}^{\alpha}(\Omega_T)$. More precisely, for any $1<\gamma<\beta:=(\alpha+\delta)/\alpha$, $\gamma\leq 2$, we have
		\begin{displaymath}
		\|\xi\|_{\alpha,\mathcal{E}}^{\alpha}\leq \gamma^*\{\|\xi\|_{L_G^{\alpha+\delta}}^{\alpha}+14^{1/\gamma}
		C_{\beta/\gamma}\|\xi\|_{L_G^{\alpha+\delta}}^{(\alpha+\delta)/\gamma}\},\quad \forall \xi\in L_{ip}(\Omega_T),
		\end{displaymath}
		where $C_{\beta/\gamma}=\sum_{i=1}^\infty i^{-\beta/\gamma}$, $\gamma^*=\gamma/(\gamma-1)$.
	\end{theorem}

    For $T>0$ and $p\geq 1$, the following spaces will be frequently used in this paper.
    \begin{itemize}
    	\item $M_G^0(0,T):=\{\eta: \eta_{t}(\omega)=\sum_{j=0}^{N-1}\xi_{j}(\omega)\textbf{1}_{[t_{j},t_{j+1})}(t),$ where  $\xi_j\in L_{ip}(\Omega_{t_j})$, $t_0\leq \cdots\leq t_N$ is a partition of $[0,T]\}$;
    	\item $M_G^p(0,T)$ is the completion of $M_G^0(0,T)$ under the norm $\Vert\eta\Vert_{M_{G}^{p}}$;
    	\item $H_G^p(0,T)$ is the completion of $M_G^0(0,T)$ under the norm $\|\eta\|_{H_G^p}$;
    	\item $S_G^0(0,T)=\{h(t,B_{t_1\wedge t}, \ldots,B_{t_n\wedge t}):t_1,\ldots,t_n\in[0,T],h\in C_{b,Lip}(\mathbb{R}^{n+1})\}$;
    	\item $S_G^p(0,T)$ is the completion of $S_G^0(0,T)$ under the norm
    	$\|\eta\|_{S_G^p}$,
    \end{itemize}
    where 
    \begin{align*}
    &\Vert\eta\Vert_{M_{G}^{p}}:=(\mathbb{\hat{E}}[\int_{0}^{T}|\eta_{s}|^{p}ds])^{1/p}, \\
    &\|\eta\|_{H_G^p}:=\{\hat{\mathbb{E}}[(\int_0^T|\eta_s|^2ds)^{p/2}]\}^{1/p},\\
     &\|\eta\|_{S_G^p}=\{\hat{\mathbb{E}}[\sup_{t\in[0,T]}|\eta_t|^p]\}^{1/p}.
     \end{align*}

   Let $\langle B\rangle$ be the quadratic variation process of the $G$-Brownian motion $B$. For two processes $\eta\in M_G^p(0,T)$ and $\zeta\in H_G^p(0,T)$, Peng established the $G$-It\^{o} integrals $\int_0^\cdot \eta_s d\langle B\rangle_s$ and $\int_0^\cdot \zeta_s dB_s$.  Similar to the classical Burkholder--Davis--Gundy inequality, the following property holds.

    \begin{proposition}[\cite{HJPS2}]\label{BDG}
	If $\eta\in H_G^{\alpha}(0,T)$ with $\alpha\geq 1$ and $p\in(0,\alpha]$, then
	$\sup_{u\in[t,T]}|\int_t^u\eta_s dB_s|^p\in L_G^1(\Omega_T)$ and
	\begin{displaymath}
	\underline{\sigma}^p c\hat{\mathbb{E}}_t[(\int_t^T |\eta_s|^2ds)^{p/2}]\leq
	\hat{\mathbb{E}}_t[\sup_{u\in[t,T]}|\int_t^u\eta_s dB_s|^p]\leq
	\bar{\sigma}^p C\hat{\mathbb{E}}_t[(\int_t^T |\eta_s|^2ds)^{p/2}],
	\end{displaymath}
	where $0<c<C<\infty$ are constants depending on $p$.
    \end{proposition}

%\begin{theorem}[\cite{S11}]\label{the1.2}
%For any $\alpha\geq 1$ and $\delta>0$, $L_G^{\alpha+\delta}(\Omega_T)\subset L_{\mathcal{E}}^{\alpha}(\Omega_T)$. More precisely, for any $1<\gamma<\beta:=(\alpha+\delta)/\alpha$, $\gamma\leq 2$, we have
%\begin{displaymath}
%\|\xi\|_{\alpha,\mathcal{E}}^{\alpha}\leq \gamma^*\{\|\xi\|_{L_G^{\alpha+\delta}}^{\alpha}+14^{1/\gamma}
%C_{\beta/\gamma}\|\xi\|_{L_G^{\alpha+\delta}}^{(\alpha+\delta)/\gamma}\},\quad \forall \xi\in L_{ip}(\Omega_T),
%\end{displaymath}
%where $C_{\beta/\gamma}=\sum_{i=1}^\infty i^{-\beta/\gamma}$, $\gamma^*=\gamma/(\gamma-1)$.
%\end{theorem}

%\begin{lemma}[\cite{BL}]\label{jensen}
%Let $h:\mathbb{R}\rightarrow\mathbb{R}$ be a continuous, non-decreasing and concave function. Then, for each $X\in L_G^1(\Omega_T)$ with $h(X)\in L_G^1(\Omega_T)$, the following inequality holds:
%$$h(\hat{\mathbb{E}}[X])\geq \hat{\mathbb{E}}[h(X)].$$
%\end{lemma}

\subsection{Reflected $G$-BSDEs with Lipschitz coefficients}

In this subsection, we present some basic results about reflected $G$-BSDEs obtained in \cite{LPSH} and \cite{LS}. We are given the following parameters: the generators $f$ and $g$, the lower obstacle $\{S_t\}_{t\in[0,T]}$ and the terminal value $\xi$.  A triple of processes $(Y,Z,A)$ is called a solution of reflected $G$-BSDE with parameters $(\xi,f,g,S)$ if for some $2\leq \alpha\leq \beta$ (where $\beta$ is a constant given in the assumption below) the following properties hold:
\begin{description}
\item[(a)]$(Y,Z,A)\in\mathcal{S}_G^{\alpha}(0,T)$ and $Y_t\geq S_t$, $0\leq t\leq T$;
\item[(b)]$Y_t=\xi+\int_t^T f(s,Y_s,Z_s)ds+\int_t^T g(s,Y_s,Z_s)d\langle B\rangle_s
-\int_t^T Z_s dB_s+(A_T-A_t)$;
\item[(c)]$\{-\int_0^t (Y_s-S_s)dA_s\}_{t\in[0,T]}$ is a nonincreasing $G$-martingale.
\end{description}
Here, we denote by $\mathcal{S}_G^{\alpha}(0,T)$ the collection of processes $(Y,Z,A)$ such that $Y\in S_G^{\alpha}(0,T)$, $Z\in H_G^{\alpha}(0,T)$, $A$ is a continuous nondecreasing process with $A_0=0$ and $A\in S_G^\alpha(0,T)$. 

Here $f$ and $g$ are maps
	\begin{displaymath}
	f(t,\omega,y,z),g(t,\omega,y,z):[0,T]\times\Omega_T\times\mathbb{R}^2\rightarrow\mathbb{R}.
	\end{displaymath}
Below, we list the assumptions on the parameters of the reflected $G$-BSDEs.

There exists some $\beta>2$ such that
	\begin{itemize}
		\item[(A1)] $|f(t,y,z)-f(t,y',z')|+|g(t,y,z)-g(t,y',z')|\leq \kappa(|y-y'|+|z-z'|)$ for some $\kappa>0$;
		\item[(A2)] for any $y,z$, $f(\cdot,\cdot,y,z)$, $g(\cdot,\cdot,y,z)\in M_G^\beta(0,T)$;
		\item[(A3)] $\{S_t\}_{t\in[0,T]}\in S_G^\beta(0,T)$ and $S_t\leq I_t$, $t\in[0,T]$, $q.s.$, where $I$ is a generalized $G$-It\^{o} process of the following form
		\begin{displaymath}
		I_t=I_0+\int_0^t b^I(s)ds+\int_0^t \sigma^I(s)dB_s+K^I_t,
		\end{displaymath}
	abd $\{b^I(t)\}_{t\in[0,T]},\{\sigma^I(t)\}_{t\in[0,T]}\in S_G^\beta(0,T)$, $K^I\in S_G^\beta(0,T)$ is a nonincreasing $G$-martingale;
		\item[(A4)] $\xi\in L_G^\beta(\Omega_T)$ and $S_T\leq \xi$, $q.s.$
	\end{itemize}
	
	The following theorem indicates that the reflected $G$-BSDEs admits a unique solution if the parameters satisfy the above conditions.
	 \begin{theorem}[\cite{LPSH},\cite{LS}]\label{the1.14}
	 	Suppose that $(\xi,f,g,S)$ satisfy \textsc{(A1)}--\textsc{(A4)}. Then, for any $2\leq \alpha<\beta$, the reflected $G$-BSDE with parameters  $(\xi,f,g,S)$ has a unique solution $(Y,Z,A)\in \mathcal{S}_G^\alpha(0,T)$. Moreover, there exists a constant $C:=C(\alpha,T, \kappa,G)>0$ such that
 	\begin{displaymath}
 	|Y_t|^\alpha\leq C\hat{\mathbb{E}}_t[|\xi|^\alpha+\int_t^T(|f(s,0,0)|^\alpha+|g(s,0,0)|^\alpha+|b^I(s)|^\alpha+|\sigma^I(s)|^\alpha )ds+\sup_{s\in[t,T]}|I_s|^\alpha].
 	\end{displaymath}	 	
	 \end{theorem}
	
	Now, we provide some a priori estimates, which will be used in Section 4.
	 \begin{proposition}[\cite{LPSH}]\label{the1.6}
	 	Let $f,g$  satisfy (A1) and (A2). Assume
	 	\begin{displaymath}
	 	Y_t=\xi+\int_t^T f(s,Y_s,Z_s)ds+\int_t^T g(s,Y_s,Z_s)d\langle B\rangle_s-\int_t^T Z_sdB_s+(A_T-A_t),
	 	\end{displaymath}
	 	where $(Y,Z,A)\in\mathcal{S}_G^\alpha(0,T)$ with $2\leq \alpha\leq \beta$. Then, there exists a constant $C:=C(\alpha, T, \kappa,G)>0$ such that for each $t\in[0,T]$,
	 	\begin{align*}
	 	\hat{\mathbb{E}}_t[(\int_t^T |Z_s|^2ds)^{\frac{\alpha}{2}}]&\leq C\{\hat{\mathbb{E}}_t[\sup_{s\in[t,T]}|Y_s|^\alpha]
	 	+(\hat{\mathbb{E}}_t[\sup_{s\in[t,T]}|Y_s|^\alpha])^{1/2}(\hat{\mathbb{E}}_t[(\int_t^T h_s ds)^\alpha])^{1/2}\},\\
	 	\hat{\mathbb{E}}_t[|A_T-A_t|^\alpha]&\leq C\{\hat{\mathbb{E}}_t[\sup_{s\in[t,T]}|Y_s|^\alpha]+\hat{\mathbb{E}}_t[(\int_t^T h_s ds)^\alpha]\},
	 	\end{align*}
	 	where $h_s=|f(s,0,0)|+|g(s,0,0)|$.
	 \end{proposition}
	
	 \begin{proposition}[\cite{LPSH}]\label{the1.7}
	 	For $i=1,2$, let $\xi^i\in L_G^{\beta}(\Omega_T)$, $f^i,g^i$ satisfy (A1) and (A2) for some $\beta>2$. Assume
	 	\begin{displaymath}
	 	Y_t^i=\xi^i+\int_t^T f^i(s,Y^i_s,Z^i_s)ds+\int_t^T g^i(s,Y_s^i,Z_s^i)d\langle B\rangle_s-\int_t^T Z_s^idB_s+(A_T^i-A_t^i),
	 	\end{displaymath}
	 	where $(Y^i,Z^i,A^i)\in\mathcal{S}_G^\alpha(0,T)$ for some $2\leq \alpha\leq \beta$. Set $\hat{Y}_t=Y^1_t-Y^2_t$, $\hat{Z}_t=Z^1_t-Z^2_t$. Then, there exists a constant $C:=C(\alpha,T,\kappa,G)$ such that
	 	\begin{align*}
	 	\hat{\mathbb{E}}[(\int_0^T|\hat{Z}_s|^2ds)^{\frac{\alpha}{2}}]\leq & C_{\alpha}\{(\hat{\mathbb{E}}[\sup_{t\in[0,T]}|\hat{Y}_t|^\alpha])^{1/2}
	 	\sum_{i=1}^2[(\hat{\mathbb{E}}[\sup_{t\in[0,T]}|{Y}^i_t|^\alpha])^{1/2}\\&+(\hat{\mathbb{E}}[(\int_0^T h_s^i ds)^\alpha])^{1/2}]+\hat{\mathbb{E}}[\sup_{t\in[0,T]}|\hat{Y}_t|^\alpha]\},
	 	\end{align*}
	 	where $h_s^i=|f^i(s,0,0)|+|g^i(s,0,0)|$.
	 \end{proposition}
	 
	 \begin{remark}
	 	Note that in the above two propositions, we do not  need to assume $(Y,Z,A)$ and $(Y^i,Z^i,A^i)$, $i=1,2$ are solutions of reflected $G$-BSDEs.
	 \end{remark}

 \begin{proposition}[\cite{LS}]\label{the1.10}
	 	Let $(\xi^1,f^1,g^1,S)$ and $(\xi^2,f^2,g^2,S)$ be two sets of data, each one satisfying the Assumptions (A1)--(A4). Let $(Y^i,Z^i,A^i)\in\mathcal{S}_G^\alpha(0,T)$ be the solutions of the reflected $G$-BSDEs with data $(\xi^i,f^i,g^i,S)$, $i=1,2$ respectively, with $2\leq \alpha\leq \beta$. Set $\hat{Y}_t=Y^1_t-Y^2_t$, $\hat{\xi}=\xi^1-\xi^2$. Then, there exists a constant $C:=C(\alpha,T, \kappa,G)>0$ such that
	 	\begin{displaymath}
	 	|\hat{Y}_t|^\alpha\leq C\hat{\mathbb{E}}_t[|\hat{\xi}|^\alpha+\int_t^T|\hat{\lambda}_s|^\alpha ds]
			 	\end{displaymath}
	 	where $\hat{\lambda}_s=|f^1(s,Y_s^2,Z_s^2)-f^2(s,Y_s^2,Z_s^2)|+|g^1(s,Y_s^2,Z_s^2)-g^2(s,Y_s^2,Z_s^2)|.$
	 \end{proposition}

\subsection{$G$-BSDEs with non-Lipschitz coefficients}

In this subsection, we recall some results on the $G$-BSDEs with non-Lipschitz coefficients obtained in \cite{He}, where the generators $f,g$ satisfy the following so-called $\beta$-order Mao's condition, i.e.,

\begin{itemize}
\item[(H1')] For any $y,y',z,z'\in\mathbb{R}$, $t\in[0,T]$,
$$|f(t,y,z)-f(t,y',z')|^\beta+|g(t,y,z)-g(t,y',z')|^\beta\leq \mu(|y-y'|^\beta)+L|z-z'|^\beta,$$
where $L>0$ is a constant and $\mu:\mathbb{R}^+\rightarrow \mathbb{R}^+$ is a continuous non-decreasing concave function with $\mu(0)=0$, $\mu(u)>0$ for $u>0$, such that $\int_{0+}\frac{du}{\mu(u)}=+\infty$ for some $\beta>2$.
\end{itemize}

 Then, we obtain the following existence and uniqueness result as well as the comparison theorem.
\begin{theorem}[\cite{He}]\label{thm3.1}
Assume that $\xi\in L_G^\beta(\Omega_T)$ and $f,g$ satisfy (H1') and (A2) for some $\beta>2$. Then, the $G$-BSDE with parameters $(\xi,f,g)$ has a unique solution $(Y,Z,K)\in \mathfrak{S}^\alpha_G(0,T)$ for any $2\leq \alpha<\beta$.
\end{theorem}

\begin{theorem}[\cite{He}]\label{thm3.7}
Let $(Y^i,Z^i,K^i)\in \mathfrak{S}_G^\alpha(0,T)$, $i=1,2$, for some $2\leq \alpha<\beta$ be the solutions of the following $G$-BSDEs
$$Y^i_t=\xi^i+\int_t^T f^i(s,Y_s^i,Z_s^i)ds+\int_t^T g^i(s,Y_s^i,Z_s^i)d\langle B\rangle_s-\int_t^T Z_s^i dB_s-(K^i_T-K^i_t),$$
where $\xi^i\in L_G^\beta(\Omega_T)$, $f^i,g^i$ satisfy (H1') and (A2) for $\beta>2$. If $\xi^1\leq \xi^2$, $f^1\leq f^2$, $g^1\leq g^2$, then $Y^1_t\leq Y^2_t$.
\end{theorem}

	%%%%%%%%%%%%%%%%%%%%%%%%%%%%%%%%%%%%%%%%%%%%%%%%%%%%%%%%%%%%%%%%%%%%%%%%%%%%%%%%%%%%%%%%%%%%%%%%%%%%%%%%%%%%%%%%%%%%%%%%%
	
	\section{Reflected $G$-BSDEs with non-Lipschitz coefficients: the uniqueness result}

In this section, we consider the reflected $G$-BSDEs with non-Lipschitz coefficients. More specifically, we assume that $f,g$ satisfy the following condition:
\begin{itemize}
\item[(H1)] For any $y,y',z,z'\in\mathbb{R}$, $t\in[0,T]$,
$$|f(t,y,z)-f(t,y',z')|+|g(t,y,z)-g(t,y',z')|\leq \rho(|y-y'|)+L|z-z'|,$$
where $L>0$ is a constant and $\rho:\mathbb{R}^+\rightarrow \mathbb{R}^+$ is a continuous non-decreasing concave function with $\rho(0)=0$, $\rho(u)>0$ for $u>0$, such that $\int_{0+}\frac{du}{\rho^\beta(u^{1/\beta})}=+\infty$ for some $\beta>2$.
\end{itemize}	
	
	Now, we state the main result in this paper.
	
	\begin{theorem}\label{main}
	Assume that $(\xi,f,g,S)$ satisfy (H1), (A2)-(A4) for some $\beta>2$. Then the reflected $G$-BSDE with parameters $(\xi,f,g,S)$ has a unique solution $(Y,Z,A)\in \mathcal{S}_G^\alpha(0,T)$ for any $2\leq \alpha<\beta$.
	\end{theorem}
	
	In this section, we mainly focus on the uniqueness result. To this end, we first list the following technical lemmas. In fact, Lemma \ref{lem2.14} is the Jensen inequality and Lemma \ref{lem2.16} is the backward Bihari inequality.
	
	\begin{lemma}[\cite{BL}]\label{lem2.14}
	Let $h(\cdot):\mathbb{R}\rightarrow\mathbb{R}$ be a continuous, nondecreasing and concave function. Then, for each $X\in L_G^1(\Omega_T)$ sucu that $h(X)\in L_G^1(\Omega_T)$, the following inequality holds:
		$$\hat{\mathbb{E}}[h(X)]\leq h(\hat{\mathbb{E}}[X]).$$
	\end{lemma}  
	
	\begin{lemma}[\cite{FJ}]\label{lem4}
	Let $h(\cdot)$ be a nondecreasing and concave function on $\mathbb{R}^+$ with $h(0)=0$. Then we have for any $r>1$, $h^r(x^{\frac{1}{r}})$ is also a nondecreasing and concave function on $\mathbb{R}^+$. Moreover, if $h(u)>0$ for $u>0$ and $\int_{0+}\frac{du}{h(u)}=\infty$, then for any $0<r<1$, we have
	$$\int_{0+}\frac{du}{h^r(u^{\frac{1}{r}})}=\infty.$$
	\end{lemma}  
	
	\begin{lemma}[\cite{M}]\label{lem2.16}
	Let $h(\cdot)$ be a continuous and nondecreasing function on $\mathbb{R}^+$ with $h(0)=0$ and $\int_{0+}\frac{du}{h(u)}=\infty$. Suppose that $u:[0,T]\rightarrow \mathbb{R}^+$ is a continuous function such that for any $t\in[0,T]$%, then for any $0<r<1$, we have
	$$u_t\leq u_0+C\int_t^T h(u_s)ds,$$
	where $C$ is a nonnegative constant. If $u_0=0$, the $u_t=0$ for all $t\in[0,T]$.
	\end{lemma} 
	
	\begin{remark}\label{rho}
It is easy to check that if $\rho$ satisfies the conditions in (H1), there exists two positive constants $a,b$ such that $\rho(u)\leq a+bu$. Besides, by Lemma \ref{lem4}, the coefficients $f,g$ satisfying (H1) also satisfies (H1'). 
\end{remark}

%\subsection{Uniqueness of reflected $G$-BSDEs with non-Lipschitz coefficients}

In order to obtain the uniqueness result of reflected $G$-BSDE with coefficients satisfying (H1) and (A2), we first establish the following a priori estimates. 
 \begin{proposition}\label{the1.6'}
	 	Let $f,g$  satisfy (H1) and (A2). Assume
	 	\begin{displaymath}
	 	Y_t=\xi+\int_t^T f(s,Y_s,Z_s)ds+\int_t^T g(s,Y_s,Z_s)d\langle B\rangle_s-\int_t^T Z_sdB_s+(A_T-A_t),
	 	\end{displaymath}
	 	where $(Y,Z,A)\in\mathcal{S}_G^\alpha(0,T)$ with $2\leq \alpha\leq \beta$. Then, there exists a constant $C:=C(\alpha, T, \rho,G)>0$ such that for each $t\in[0,T]$,
	 	\begin{align*}
	 	\hat{\mathbb{E}}_t[(\int_t^T |Z_s|^2ds)^{\frac{\alpha}{2}}]&\leq C\{1+\hat{\mathbb{E}}_t[\sup_{s\in[t,T]}|Y_s|^\alpha]
	 	+(\hat{\mathbb{E}}_t[\sup_{s\in[t,T]}|Y_s|^\alpha])^{1/2}(\hat{\mathbb{E}}_t[(\int_t^T h_s ds)^\alpha])^{1/2}\},\\
	 	\hat{\mathbb{E}}_t[|A_T-A_t|^\alpha]&\leq C\{1+\hat{\mathbb{E}}_t[\sup_{s\in[t,T]}|Y_s|^\alpha]+\hat{\mathbb{E}}_t[(\int_t^T h_s ds)^\alpha]\},
	 	\end{align*}
	 	where $h_s=|f(s,0,0)|+|g(s,0,0)|$.
	 \end{proposition}
	 
	 \begin{remark}
	 Compared with Proposition \ref{the1.6}, there is an additional constant in our inequality due to the property of $\rho$ (see Remark \ref{rho}).
	 \end{remark}
	 
	 \begin{proof}
	 For simplicity, we only prove the case that $g\equiv 0$ and $t=0$. Applying $G$-It\^{o}'s formula to $|Y_t|^2$, we have
	 \begin{displaymath}
	 |Y_0|^2+\int_0^T |Z_s|^2d\langle B\rangle_s=|\xi|^2+\int_0^T 2Y_sf_s ds-\int_0^T 2Y_sZ_s dB_s+\int_0^T 2Y_s dA_s,
	 \end{displaymath}
	 where $f_s=f(s,Y_s,Z_s)$. Comparing with the proof of Proposition 3.5 in \cite{HJPS1}, the main difference is the calculation for $\hat{\mathbb{E}}[|\int_0^T Y_s f_s ds|^{\alpha/2}]$. Recalling Remark \ref{rho} and using the fact that $2|a|\leq 1+a^2$, we obtain that 
	 \begin{align*}
	& \hat{\mathbb{E}}[|\int_0^T Y_s f_s ds|^{\alpha/2}]\\
	\leq &C\hat{\mathbb{E}}[|\int_0^T |Y_s|(\rho(|Y_s|)+f^0_s+|Z_s|)ds|^{\alpha/2}]\\
	 \leq & C\hat{\mathbb{E}}[|\int_0^T (|Y_s|+|Y_s|^2+|Y_sf_s^0|+|Y_sZ_s|)ds|^{\alpha/2}]\\
	 \leq &C\{1+\hat{\mathbb{E}}[\sup_{s\in[0,T]}|Y_s|^\alpha]+(\hat{\mathbb{E}}[\sup_{s\in[0,T]}|Y_s|^\alpha])^{\frac{1}{2}} [ (\hat{\mathbb{E}}[(\int_0^T f_s^0 ds)^\alpha])^{\frac{1}{2}}  +(\hat{\mathbb{E}}[(\int_0^T |Z_s|^2ds)^{\frac{\alpha}{2}}])^{\frac{1}{2}}\},
	 \end{align*}
      where $f^0_s=|f(s,0,0)|$.	 Then, we have 
	 \begin{equation}\begin{split}\label{3.5}
	 \hat{\mathbb{E}}[(\int_0^T |Z_s|^2ds)^{\frac{\alpha}{2}}]\leq &C\{  1+\hat{\mathbb{E}}[\sup_{s\in[0,T]}|Y_s|^\alpha]+(\hat{\mathbb{E}}[\sup_{s\in[0,T]}|Y_s|^\alpha])^{\frac{1}{2}}\\
	& \times [ (\hat{\mathbb{E}}[(\int_0^T f_s^0 ds)^\alpha])^{\frac{1}{2}}  +(\hat{\mathbb{E}}[|A_T|^\alpha])^{\frac{1}{2}}  ]\}.
	 \end{split}\end{equation}
	 On the other hand, 
	 $$A_T=Y_0-\xi-\int_0^T f_s ds+\int_0^T Z_s dB_s.$$
	 Using the property of $\rho$, simple calculation implies that 
	 \begin{equation}\label{3.6}
	 \hat{\mathbb{E}}[|A_T|^\alpha]\leq C\{ 1+\hat{\mathbb{E}}[\sup_{s\in[0,T]}|Y_s|^\alpha]+  \hat{\mathbb{E}}[(\int_0^T f_s^0 ds)^\alpha] + \hat{\mathbb{E}}[(\int_0^T |Z_s|^2ds)^{\frac{\alpha}{2}}] \}
	 \end{equation}
	 By \eqref{3.5} and \eqref{3.6}, we obtain the desired result.
	 \end{proof}

	\begin{proof}[Proof of Theorem \ref{main}: the uniqueness result]
	For simplicity, we assume that $g\equiv 0$. Suppose that $(Y^i,Z^i,A^i)\in\mathcal{S}^\alpha(0,T)$, $i=1,2$ are solutions to the reflected $G$-BSDE with parameters $(\xi,f,S)$.  Set $\hat{f}_t=f(t,Y_t^1,Z_t^1)-f(t,Y_t^2,Z_t^2)$, $\hat{A}_t=A^1_t-A^2_t$, $\tilde{Z}_t=Z^1_t-Z^2_t$ and $\bar{Y}_t=|\tilde{Y}_t|^2=|Y^1_t-Y^2_t|^2$. Applying It\^{o}'s formula to $\bar{Y}_t^{\frac{\alpha}{2}}e^{rt}$, where $r>0$ will be determined later, we get
\begin{equation}\label{0}
\begin{split}
&\quad \bar{Y}_t^{\alpha/2}e^{rt}+\int_t^T re^{rs}\bar{Y}_s^{\alpha/2}ds+\int_t^T \frac{\alpha}{2} e^{rs}
\bar{Y}_s^{\alpha/2-1}(\tilde{Z}_s)^2d\langle B\rangle_s\\
&=\alpha(1-\frac{\alpha}{2})\int_t^Te^{rs}\bar{Y}_s^{\alpha/2-2}(\tilde{Y}_s)^2(\tilde{Z}_s)^2d\langle B\rangle_s-\int_t^T\alpha e^{rs}\bar{Y}_s^{\alpha/2-1}\tilde{Y}_s\tilde{Z}_sdB_s\\
&\quad+\int_t^T{\alpha} e^{rs}\bar{Y}_s^{\alpha/2-1}\tilde{Y}_s\hat{f}_sds +\int_t^T\alpha e^{rs}\bar{Y}_s^{\alpha/2-1}\tilde{Y}_sd\hat{A}_s\\
&\leq\int_t^T{\alpha}e^{rs}\bar{Y}_s^{\frac{\alpha-1}{2}}|\hat{f}_s|ds
%\int_t^T\alpha e^{rs}\bar{Y}_s^{\alpha/2-1}\hat{Y}_s\hat{Z}_sdB_s\\
%&\quad+\int_t^T\alpha e^{rs}\bar{Y}_s^{\alpha/2-1}(\hat{Y}_s)^-dA^2_s+\int_t^T\alpha e^{rs}\bar{Y}_s^{\alpha/2-1}(\hat{Y}_s)^+dA^1_s\\
+\alpha(1-\frac{\alpha}{2})\int_t^Te^{rs}\bar{Y}_s^{\alpha/2-1}(\tilde{Z}_s)^2d\langle B\rangle_s-(M_T-M_t),
%&\quad+\int_t^T{\alpha} e^{rs}\bar{Y}_s^{\frac{\alpha-1}{2}}|f^1(s,Y_s^2,Z_s^2)-f^2(s,Y_s^2,Z_s^2)|ds,
\end{split}
\end{equation}
where $M_t=\int_0^t \alpha e^{rs}\bar{Y}_s^{\alpha/2-1}(\tilde{Y}_s\tilde{Z}_sdB_s-(\tilde{Y}_s)^+dA_s^1-(\tilde{Y}_s)^-dA_s^2)$. By the proof of Proposition 3.4 in \cite{LPSH}, $\{M_t\}_{t\in[0,T]}$ is a $G$-martingale. %For readers' convenience, we give a short proof here. Indeed, note that
%\begin{displaymath}
%\tilde{Y}_t=Y^1_t-S_t+S_t-Y^2_t\leq Y^1_t-S_t.
%\end{displaymath}
%Consequently, 
%\begin{displaymath}
%(\tilde{Y}_t)^+\leq (Y^1_t-S_t)^+=Y^1_t-S_t.
%\end{displaymath}
%Then, we obtain 
%\begin{displaymath}
%0\geq -\int_t^T (\tilde{Y}_s)^+ dA^1_s\geq -\int_t^T (Y^1_s-S_s)dA^1_s.
%\end{displaymath}
%Thus, we can conclude that
%\begin{displaymath}
%0\geq \hat{\mathbb{E}}_t[ -\int_t^T (\tilde{Y}_s)^+ dA^1_s]\geq \hat{\mathbb{E}}_t[-\int_t^T (Y^1_s-S_s)dA^1_s]=0.
%\end{displaymath}
%It follows that the process $\{K_t^1\}=\{-\int_0^t (\tilde{Y}_s)^+ dA^1_s\}$ is a nonincreasing $G$-martingale. Set $K_t^2=-\int_0^t (\tilde{Y}_s)^- dA^2_s$. Both $\{K_t^1\}$ and $\{K_t^2\}$ are nonincreasing $G$-martingales, so is $\{\int_0^t \alpha e^{rs}\bar{Y}_s^{\alpha/2-1}(dK_s^1+dK_s^2)\}$, which yields that $\{M_t\}_{t\in[0,T]}$ is a $G$-martingale.
From the assumption of $f^1$, we derive that
\begin{equation}\label{1}
\begin{split}
\int_t^T{\alpha} e^{rs}\bar{Y}_s^{\frac{\alpha-1}{2}}|\hat{f}_s|ds
&\leq \int_t^T{\alpha} e^{rs}\bar{Y}_s^{\frac{\alpha-1}{2}}\{ \rho(|\tilde{Y}_s|)+L|\tilde{Z}_s|  \}ds\\
&\leq \frac{\alpha L^2}{\underline{\sigma}^2(\alpha-1)}\int_t^T e^{rs}\bar{Y}_s^{\alpha/2}ds
+\int_t^T{\alpha} e^{rs}\bar{Y}_s^{\frac{\alpha-1}{2}}\rho(|\tilde{Y}_s|)ds\\
&\quad+\frac{\alpha(\alpha-1)}{4}\int_t^Te^{rs}\bar{Y}_s^{\alpha/2-1}(\tilde{Z}_s)^2d\langle B\rangle_s.
%&\quad +3(\alpha-1)\int_t^T e^{rs}\bar{Y}_s^{\alpha/2}ds+\int_t^T e^{rs}(L^\alpha|S_s^1-S_s^2|^\alpha+L^\alpha|\sigma^1(s)-\sigma^2(s)|^\alpha+|b^1(s)-b^2(s)|^\alpha)ds.
\end{split}
\end{equation}
By Young's inequality, we have
\begin{equation}\label{2}
%\begin{split}
\int_t^T{\alpha} e^{rs}\bar{Y}_s^{\frac{\alpha-1}{2}}\rho(|\tilde{Y}_s|)ds\\
 \leq (\alpha-1)\int_t^T  e^{rs}\bar{Y}_s^{\alpha/2}ds+\int_t^T e^{rs}\rho^\alpha(|\tilde{Y}_s|)ds.
%\end{split}
\end{equation}
By \eqref{0}--\eqref{2} and setting $r=\alpha+\frac{\alpha L^{2}}{\underline{\sigma}^2(\alpha-1)}$, we get
\begin{displaymath}
\bar{Y}_t^{\alpha/2}e^{rt}+(M_T-M_t)\leq C\int_t^Te^{rs}\rho^{\alpha}(|\tilde{Y}_s|)ds.
\end{displaymath}
Taking conditional expectation on both sides, we have
\begin{align}\label{problem}
|\tilde{Y}_t|^\alpha\leq C\hat{\mathbb{E}}_t[\int_t^T\tilde{\rho}(|\tilde{Y}_s|^\alpha)ds],
\end{align}
where $\tilde{\rho}(x)=\rho^{\alpha}(|x|^{\frac{1}{\alpha}})$. By Lemma \ref{lem4}, $\tilde{\rho}$ is a continuous, nondecreasing and concave function. Taking expectations on both sides and applying Lemma \ref{lem2.14} yield that 
\begin{displaymath}
\hat{\mathbb{E}}[|\tilde{Y}_t|^\alpha]\leq C\hat{\mathbb{E}}[\int_t^T\tilde{\rho}(|\tilde{Y}_s|^\alpha)ds]\leq  C\int_t^T \tilde{\rho}(\hat{\mathbb{E}}[|\tilde{Y}_s|^\alpha])ds.
\end{displaymath}
By Lemma \ref{lem4} again, noting that $\alpha<\beta$ and $\int_{0+}\frac{du}{\rho^\beta(u^{1/\beta})}=+\infty$ , we have $\int_{0+}\frac{du}{\tilde{\rho}(u)}=\infty$. Applying Lemma \ref{lem2.16}, we obtain that $\hat{\mathbb{E}}[|\tilde{Y}_t|^\alpha]=0$, for any $t\in[0,T]$. Since $Y^i$, $i=1,2$ are continuous, by Lemma 2.19 in \cite{He}, $Y^1_t=Y^2_t$, $t\in[0,T]$, q.s. 

Now, we prove that $\hat{\mathbb{E}}[(\int_0^T |\tilde{Z}_s|^2 ds)^{\alpha/2}]=0$. Taking $\alpha=2$ and $r=0$ in \eqref{0} implies that
\begin{displaymath}
|\tilde{Y}_0|^2+\int_0^T (\tilde{Z}_s)^2 d\langle B\rangle_s=\int_0^T 2\tilde{Y}_s\hat{f}_s ds-\int_0^T 2\tilde{Y}_s\tilde{Z}_s dB_s+\int_0^T 2\tilde{Y}_s d\hat{A}_s.
\end{displaymath}
By the property of $\rho$, we obtain that
\begin{align*}
\hat{\mathbb{E}}[(\int_0^T \tilde{Y}_s \hat{f}_s ds)^{\frac{\alpha}{2}}]\leq &\hat{\mathbb{E}}[(\int_0^T |\tilde{Y}_s|(\rho(|\tilde{Y}_s|)+L|\tilde{Z}_s|)ds)^{\frac{\alpha}{2}}]\\
\leq & C\{\hat{\mathbb{E}}[\sup_{s\in[0,T]}|\tilde{Y}_s|^{\frac{\alpha}{2}}+\sup_{s\in[0,T]}|\tilde{Y}_s|^\alpha]+(\hat{\mathbb{E}}[\sup_{s\in[0,T]}|\tilde{Y}_s|^\alpha])^{\frac{1}{2}}(\hat{\mathbb{E}}[(\int_0^T|\tilde{Z}_s|^2 ds)^{\frac{\alpha}{2}}])^{\frac{1}{2}}\}.
\end{align*}
Using BDG inequality (see Proposition \ref{BDG}), we have
\begin{align*}
\hat{\mathbb{E}}[(\int_0^T \tilde{Y}_s\tilde{Z}_s dB_s)^{\frac{\alpha}{2}}]\leq C(\hat{\mathbb{E}}[\sup_{s\in[0,T]}|\tilde{Y}_s|^\alpha])^{\frac{1}{2}}(\hat{\mathbb{E}}[(\int_0^T|\tilde{Z}_s|^2 ds)^{\frac{\alpha}{2}}])^{\frac{1}{2}}\}.
\end{align*}
By Proposition \ref{the1.6'}, $\hat{\mathbb{E}}[|A_T^i|^\alpha]$ are bounded $i=1,2$. It follows that 
\begin{align*}
\hat{\mathbb{E}}[ (\int_0^T \tilde{Y}_s d\hat{A}_s)^{\frac{\alpha}{2}} ]\leq C(\hat{\mathbb{E}}[\sup_{s\in[0,T]}|\tilde{Y}_s|^\alpha])^{\frac{1}{2}}.
\end{align*}
All the above analysis indicates that 
\begin{align*}
\hat{\mathbb{E}}[(\int_0^T |\tilde{Z}_s|^2 ds)^{\frac{\alpha}{2}}]\leq C\{  (\hat{\mathbb{E}}[\sup_{s\in[0,T]}|\tilde{Y}_s|^\alpha])^{\frac{1}{2}}+\hat{\mathbb{E}}[\sup_{s\in[0,T]}|\tilde{Y}_s|^\alpha]  \},
\end{align*}
which implies that $Z^1\equiv Z^2$.

Since
\begin{displaymath}
\hat{A}_t=\hat{Y}_0-\int_0^t \hat{f}_s ds+\int_0^t \tilde{Z}_s dB_s,
\end{displaymath}
it is easy to check that $\hat{\mathbb{E}}[\sup_{s\in[0,T]}|\hat{A}_s|^\alpha]=0$. The proof is complete.
\end{proof}

\begin{remark}
If we apply Assumption (H1') instead of (H1), then \eqref{problem} turns into
$$|\tilde{Y}_t|^\alpha\leq C\hat{\mathbb{E}}_t[\int_t^T\tilde{\mu}(|\tilde{Y}_s|^\alpha)ds],$$
where $\tilde{\mu}(x)=\mu^{\alpha/\beta}(x^{\beta/\alpha})$. However, since $\alpha/\beta<1$, $\tilde{\mu}$ may not be a concave function. Therefore, we cannot apply Lemma \ref{lem2.14} and Lemma \ref{lem2.16} to get the uniqueness result.
\end{remark}

	\section{Reflected $G$-BSDEs with non-Lipschitz coefficients: the existence result}
	
	In this section, we establish the existence result for reflected $G$-BSDEs with non-Lipschitz coefficients using two approaches: the Picard iteration and approximation via penalization. For simplicity, we assume that $g\equiv 0$. The proof is effective for the other cases.
	
	\subsection{Construction via Picard iteration}
	
	In this subection, we construct the solution by a Picard iteration method.  Let $Y^0\equiv 0$. Consider the family of reflected $G$-BSDEs parameterized by $n=1,2,\cdots$,
	\begin{displaymath}
	\begin{cases}
	Y^n_t=\xi+\int_t^T f(s,Y^{n-1}_s,Z^n_s)ds-\int_t^T Z^n_s dB_s+(A^n_T-A^n_t),\\
Y^n_t\geq S_t, t\in[0,T],\\
 \{-\int_0^t (Y^n_s-S_s)dA^n_s\}_{t\in[0,T]} \textrm{ is a nonincreasing $G$-martingale.}
\end{cases}
        \end{displaymath}
        
        We first recall the following lemma obtained in \cite{He}, which makes the integrand $f(\cdot,\cdot,Y^n_\cdot,z)$ will defined for each fixed $z$.
        
        \begin{lemma}[\cite{He}]\label{lem2.15}
        For some $p>1$, suppose that $f(\cdot,\cdot,x)$ is a given process satisfying $f(\cdot,\cdot,x)\in M^p_G(0,T)$ for each $x\in\mathbb{R}$. Moreover, we assume that for any $x,x'\in\mathbb{R}$ and $t\in[0,T]$, 
        $$|f(t,x)-f(t,x')|\leq \gamma(|x-x'|),$$
        where $\gamma:[0,+\infty)\rightarrow[0,+\infty)$ is a nondecreasing and concave function vanishing at $0$. Then, for any $X\in M_G^p(0,T)$, we have $f(\cdot,\cdot,X_\cdot)\in M_G^p(0,T)$. 
        \end{lemma}
        
        By Lemma \ref{lem2.15}, for any fixed $z$, $f(\cdot,\cdot,0,z)\in M_G^\beta(0,T)$. Then by Theorem \ref{the1.14}, for any $2\leq \alpha<\beta$, we have $Y^1\in S^\alpha_G(0,T)$. Hence, for any $\varepsilon\in(0,\beta-2)$, we have $Y^1\in S_G^{\beta-\frac{\varepsilon}{2}}(0,T)\subset S_G^{\beta-{\varepsilon}}(0,T)$, which implies that $f(\cdot,\cdot,Y^1,z)\in M_G^{\beta-\frac{\varepsilon}{2}}(0,T)$ by Lemma \ref{lem2.15}. Following this procedure, using Theorem \ref{the1.14} and Lemma \ref{lem2.15},  for any fixed $z$, we have $f(\cdot,\cdot,Y^{n-1},z)\in M_G^{\beta-(1-\frac{1}{2^{n-1}})\varepsilon}(0,T)$ and  $Y^n\in S_G^{\beta-(1-\frac{1}{2^{n}})\varepsilon}(0,T)\subset S_G^{\beta-\varepsilon}(0,T)$, $n=2,3,\cdots$. Since $\varepsilon$ can be arbitrarily small, for any $2\leq \alpha<\beta$ and $n\in\mathbb{N}$, we have $Y^n\in S_G^\alpha(0,T)$. Actually, we have the following uniform estimates for $Y^n$.
        
        \begin{lemma}\label{lem3.3}
        Assume that $(\xi,f,g,S)$ satisfy (H1), (A2)-(A4) for some $\beta>2$. There exists a constant $C$ depending on $\beta,T,G,L,\rho$ but not $n$, such that 
        \begin{displaymath}
        \sup_{t\in[0,T]}\hat{\mathbb{E}}[|Y_t^n|^\beta]\leq C.
        \end{displaymath}
        \end{lemma}
	
	\begin{proof}
	The proof is similar with the one of Lemma 3.3 in \cite{He}. For readers' convenience, we give a short proof here. By Theorem \ref{the1.14}, for any $2\leq \alpha<\beta$, we have
	\begin{displaymath}
 	|Y_t^n|^\alpha\leq C\hat{\mathbb{E}}_t[|\xi|^\alpha+\int_t^T(|f(s,Y^{n-1}_s,0)|^\alpha+|b^I(s)|^\alpha+|\sigma^I(s)|^\alpha )ds+\sup_{s\in[t,T]}|I_s|^\alpha].
 	\end{displaymath}	
	We then obtain the following equation by the H\"{o}lder inequality
	\begin{displaymath}
 	|Y_t^n|^\beta\leq C\hat{\mathbb{E}}_t[|\xi|^\beta+\int_t^T(|f(s,Y^{n-1}_s,0)|^\beta+|b^I(s)|^\beta+|\sigma^I(s)|^\beta )ds+\sup_{s\in[t,T]}|I_s|^\beta].
 	\end{displaymath}	
	Recalling Remark \ref{rho} and the assumption on $f$, we have%Since $\rho$ is concave with $\rho(0)=0$, there exist two positive constants $a,b$ such that $\rho(u)\leq a+bu$ for any $u\geq 0$. Therefore, we have 
	$$|f(s,Y^{n-1}_s,0)|^\beta\leq C(1+|f(s,0,0)|^\beta+|Y^{n-1}_s|^\beta).$$
	All the above analysis yields that 
	\begin{align*}
	\hat{\mathbb{E}}[|Y^n_t|^\beta]&\leq C\{1+\hat{\mathbb{E}}[|\xi|^\beta+\int_0^T(|b^I(s)|^\beta+|\sigma^I(s)|^\beta )ds+\sup_{s\in[0,T]}|I_s|^\beta]+\hat{\mathbb{E}}[\int_t^T |Y^{n-1}_s|^\beta ds]\}\\
	&\leq C\{1+\int_t^T \hat{\mathbb{E}}[|Y^{n-1}_s|^\beta]ds\}.
	\end{align*}
	Let $p(\cdot)$ be the solution of the ODE $p(t)=C(1+\int_t^T p(s)ds)$. It is easy to verify that $p(t)=Ce^{C(T-t)}$. By an induction method, for any $n\geq 1$ and $t\in[0,T]$, we have 
	$$\hat{\mathbb{E}}[|Y^n_t|^\beta]\leq p(t)\leq Ce^{CT},$$
	which completes the proof.
	\end{proof}
	
	The following lemma indicates that $\{Y^n\}_{n\in\mathbb{N}}$ is a Cauchy sequence under the norm $\|\cdot\|_{\bar{S}_G^\beta}$, where $\|\eta\|_{\bar{S}_G^\beta}:=\sup_{t\in[0,T]}\hat{\mathbb{E}}[|\eta_t|^\beta]^{\frac{1}{\beta}}$.
	\begin{lemma}\label{lem3.4}
	Assume that $(\xi,f,g,S)$ satisfy (H1), (A2)-(A4) for some $\beta>2$. Then, we have
	\begin{displaymath}
	\lim_{n,m\rightarrow\infty} \sup_{t\in[0,T]}\hat{\mathbb{E}}[|Y^{n+m}_t-Y^n_t|^\beta]=0.
	\end{displaymath}
	\end{lemma}
	
	\begin{proof}
	The proof needs Proposition \ref{the1.10} and the analysis is similar to the proof of Lemma 3.4 in \cite{He}. We omit it.
	\end{proof}
	
	\begin{lemma}\label{lem3.5}
	Assume that $(\xi,f,g,S)$ satisfy (H1), (A2)-(A4) for some $\beta>2$. For any $2\leq \alpha<\beta$, there exists a constant $C$ depending on $\alpha,\beta,T,G,L,\rho$ but not $m,n$, such that
	\begin{align*}
	%&\hat{\mathbb{E}}[|A_T^n|^\alpha]+\hat{\mathbb{E}}[(\int_0^T|Z^n_t|^{\frac{\alpha}{2}}dt)^2]\leq  C\bigg\{1+\hat{\mathbb{E}}[\sup_{t\in[0,T]}|Y^n_t|^\alpha]+\hat{\mathbb{E}}[\sup_{t\in[0,T]}|Y^{n-1}_t|^\alpha]+\hat{\mathbb{E}}[\int_0^T|f(s,0,0)|^\alpha ds]\bigg\},\\
	%&\hat{\mathbb{E}}[(\int_0^T|Z^{n,m}_t|^{\frac{\alpha}{2}}dt)^2]\leq C\Bigg\{\hat{\mathbb{E}}[\sup_{t\in[0,T]}|Y^{n,m}_t|^\alpha]+(\hat{\mathbb{E}}[\sup_{t\in[0,T]}|Y^{n,m}_t|^\alpha])^{\frac{1}{2}}\bigg[ (\hat{\mathbb{E}}[\sup_{t\in[0,T]}|Y^n_t|^\alpha)^{\frac{1}{2}}\\
	% & \ \   +(\hat{\mathbb{E}}[\sup_{t\in[0,T]}|Y^{n-1}_t|^\alpha)^{\frac{1}{2}}+(\hat{\mathbb{E}}[\sup_{t\in[0,T]}|Y^{n+m}_t|^\alpha)^{\frac{1}{2}}+(\hat{\mathbb{E}}[\sup_{t\in[0,T]}|Y^{n+m-1}_t|^\alpha)^{\frac{1}{2}}+(\hat{\mathbb{E}}[\int_0^T|f(s,0,0)|^\alpha ds])^{\frac{1}{2}}\bigg] \Bigg\},\\
	 \|Z^n\|_{H^\alpha_G}^\alpha+\|A^n_T\|^\alpha_{L^\alpha_G}\leq &C\left\{1+\|Y^n\|^\alpha_{S_G^\alpha}+\|Y^{n-1}\|^\alpha_{S_G^\alpha}+\hat{\mathbb{E}}[\int_0^T|f(s,0,0)|^\alpha ds]\right\},\\
	 \|Z^{n,m}\|_{H_G^\alpha}^\alpha\leq &C\Bigg\{\|Y^{n,m}\|^\alpha_{S_G^\alpha} +\|Y^{n,m}\|^{\frac{\alpha}{2}}_{S_G^\alpha}\bigg[ 1+ \|Y^{n-1}\|^{\frac{\alpha}{2}}_{S_G^\alpha} + \|Y^{n}\|^{\frac{\alpha}{2}}_{S_G^\alpha}\\
	 &+ \|Y^{n+m-1}\|^{\frac{\alpha}{2}}_{S_G^\alpha}+ \|Y^{n+m}\|^{\frac{\alpha}{2}}_{S_G^\alpha}+(\hat{\mathbb{E}}[\int_0^T|f(s,0,0)|^\alpha ds])^{\frac{1}{2}}\bigg]\Bigg\}
	\end{align*}
	where $Y^{n,m}_t=Y^{n+m}_t-Y^n_t$, $Z^{n,m}_t=Z^{n+m}_t-Z^n_t$.
	\end{lemma}
	
	\begin{proof}
	By a similar analysis as the proof of Lemma \ref{lem3.3}, there exists a constant $C$ independent of $n$, such that 
	$$|f(s,Y^{n-1}_s,0)|^\alpha\leq C(1+|f(s,0,0)|^\alpha+|Y^{n-1}_s|^\alpha).$$
	Then, applying Proposition \ref{the1.6} and \ref{the1.7}, we get the desired result.
	\end{proof}
	
	\begin{proof}[Proof of Theorem \ref{main}: the existence result]
	The proof will be divided into the following steps.
	
	\textbf{Step 1.} We claim that there exists a constant $C$ independent of $n$, such that for any $2\leq \alpha<\beta$
	$$\hat{\mathbb{E}}[\sup_{t\in[0,T]}|Y^n_t|^\alpha]\leq C.$$
	Set $\mathbb{I}=|\xi|^\alpha+\int_0^T( |b^I(s)|^\alpha+|\sigma^I(s)|^\alpha )ds+\sup_{t\in[0,T]}|I_t|^\alpha$. Clearly, we have $\hat{\mathbb{E}}[\mathbb{I}^{\frac{\beta}{\alpha}}]\leq C$. By Theorem \ref{the1.14}, it follows that 
	\begin{align*}
	\hat{\mathbb{E}}[\sup_{t\in[0,T]}|Y^n_t|^\alpha]\leq C\hat{\mathbb{E}}[\sup_{t\in[0,T]}\hat{\mathbb{E}}_t[\mathbb{I}+\int_t^T |f(s,Y^{n-1}_s,0)|^\alpha ds]].
	%\leq &C\hat{\mathbb{E}}[\sup_{t\in[0,T]}\hat{\mathbb{E}}_t[1+\mathbb{I}+\int_0^T |f(s,0,0)|^\alpha ds]]
	\end{align*} 
	Since $\alpha<\beta$, by Theorem \ref{the1.2}, it is sufficient to prove that there exists a constant $C$ independent of $n$, such that 
	$$\hat{\mathbb{E}}[|\mathbb{I}+\int_t^T |f(s,Y^{n-1}_s,0)|^\alpha ds|^{\frac{\beta}{\alpha}}]\leq C.$$ 
	Indeed, we have 
	\begin{align*}
	&\hat{\mathbb{E}}[|\mathbb{I}+\int_t^T |f(s,Y^{n-1}_s,0)|^\alpha ds|^{\frac{\beta}{\alpha}}]\\
	\leq &C\hat{\mathbb{E}}[1+\mathbb{I}^{\frac{\beta}{\alpha}}+\int_0^T |Y^{n-1}_s|^\beta ds+\int_0^T |f(s,0,0)|^\beta ds]\\
	\leq &C\hat{\mathbb{E}}[1+\mathbb{I}^{\frac{\beta}{\alpha}}+\int_0^T |f(s,0,0)|^\beta ds]+\int_0^T\hat{\mathbb{E}}[ |Y^{n-1}_s|^\beta] ds\leq C,
		\end{align*}
		where we have used Lemma \ref{lem3.3} in the last inequality. Hence, the claim holds true. Consequently, by Lemma \ref{lem3.5}, there exists a constant $C$ independent of $n$, such that 
		$$\hat{\mathbb{E}}[|A_T^n|^\alpha]+\hat{\mathbb{E}}[(\int_0^T|Z^n_t|^{\frac{\alpha}{2}}dt)^2]\leq C.$$
		
       \textbf{Step 2.} We show that $\{Y^n\}_{n\in\mathbb{N}}$, $\{Z^n\}_{n\in\mathbb{N}}$, $\{A^n\}_{n\in\mathbb{N}}$ are Cauchy sequences in $S_G^\alpha(0,T)$, $H^\alpha_G(0,T)$ and $S^\alpha_G(0,T)$, respectively. Set $Y^{n,m}_t=Y^{n+m}_t-Y^n_t$. Applying Proposition \ref{the1.10} implies that 
       $$\hat{\mathbb{E}}[\sup_{t\in[0,T]}|Y^{n,m}_t|^\alpha]\leq C\hat{\mathbb{E}}[\sup_{t\in[0,T]}\hat{\mathbb{E}}_t[\int_0^T |f(s,Y^{n+m-1}_s,Z^{n+m}_s)-f(s,Y^{n-1}_s,Z^{n+m}_s)|^\alpha ds]].$$
       By Lemma \ref{lem4}, the function $\widetilde{\rho}(x):=\rho^\beta(x^{\frac{1}{\beta}})$ is a continuous, nondecreasing and concave function. Apply the H\"{o}lder inequality and Lemma \ref{lem2.14}, we obtain that 
       \begin{align*}
       &\hat{\mathbb{E}}[|\int_0^T |f(s,Y^{n+m-1}_s,Z^{n+m}_s)-f(s,Y^{n-1}_s,Z^{n+m}_s)|^\alpha ds|^{\frac{\beta}{\alpha}}]\\
       \leq &C\hat{\mathbb{E}}[\int_0^T |f(s,Y^{n+m-1}_s,Z^{n+m}_s)-f(s,Y^{n-1}_s,Z^{n+m}_s)|^\beta ds]\\
       \leq &C\hat{\mathbb{E}}[\int_0^T \rho^\beta(|Y^{n,m}_s|)ds]= C\hat{\mathbb{E}}[\int_0^T\widetilde{\rho}(|Y^{n,m}_s|^\beta)ds]\\
       \leq &C\widetilde{\rho}(\sup_{s\in[0,T]}\hat{\mathbb{E}}[|Y^{n,m}_s|^\beta]).
       \end{align*} 
       By Theorem \ref{the1.2} and Lemma \ref{lem3.4}, all the above analysis indicates that 
       $$\lim_{n,m\rightarrow\infty}\hat{\mathbb{E}}[\sup_{t\in[0,T]}|Y^{n,m}_t|^\alpha]=0.$$
       Hence, there exists a process $Y\in S_G^\alpha(0,T)$, such that 
        $$\lim_{n\rightarrow\infty}\hat{\mathbb{E}}[\sup_{t\in[0,T]}|Y^n_t-Y_t|^\alpha]=0.$$
        It follows from Lemma \ref{lem3.5} that $\{Z^n\}_{n\in\mathbb{N}}$ is a Cauchy sequences in  $H^\alpha_G(0,T)$. Consequently, there exists a process $Z\in H_G^\alpha(0,T)$, such that 
         $$\lim_{n\rightarrow \infty}\hat{\mathbb{E}}[(\int_0^T |Z^n_t-Z_t|^2dt)^{\frac{\alpha}{2}}]=0.$$
         Set 
         $$A_t=Y_0-Y_t-\int_0^tf(s,Y_s,Z_s)ds+\int_0^t Z_s dB_s.$$
         Applying a similar analysis as the proof of Theorem 3.1 (Step 3.) in \cite{He}, we have 
         $$\lim_{n\rightarrow\infty}\|A^n-A\|_{S^\alpha_G}=0.$$
         
         \textbf{Step 3.} It remains to prove that $Y_t\geq S_t$, $t\in[0,T]$ and  $\{-\int_0^t (Y_s-S_s)dA_s\}_{t\in[0,T]}$  is a nonincreasing $G$-martingale. The first assertion is clear and the proof of the second assertion is similar as the proof for Theorem 5.1 in \cite{LPSH}.
	\end{proof}

%%%%%%%%%%%%%%%%%%%%%%%%%%%%%%%%%%%%%%%%%%%%%%%%%%%%%%%%%%%%%

\subsection{Construction via penalization}

In this section, we will construct the solutions for reflected $G$-BSDEs with non-Lipschitz coefficients by a penalization method.  For $n\in\mathbb{N}$,  consider the following family of $G$-BSDEs
	\begin{equation}\label{equ1}
%\begin{split}
	Y_t^n=\xi+\int_t^T f(s,Y_s^n,Z_s^n)ds-\int_t^T Z_s^ndB_s-(K_T^n-K_t^n)+n\int_t^T(Y_s^n-S_s)^-ds.
%\end{split}	
\end{equation}
	By Theorem \ref{thm3.1}, for any $n\in\mathbb{N}$ and $2\leq \alpha<\beta$, there exists a unique solution $(Y^n,Z^n,K^n)\in \mathfrak{S}^\alpha_G(0,T)$ to the above equation. 
	Now let $L_t^{n}=n\int_0^t (Y_s^n-S_s)^-ds$, $A^n_t=L^n_t-K^n_t$. Then $\{A_t^{n}\}_{t\in[0,T]}$ is a nondecreasing process. We can rewrite $G$-BSDE \eqref{equ1} as
	\begin{equation}\label{equ}
%\begin{split}
	Y_t^n=\xi+\int_t^T f(s,Y_s^n,Z_s^n)ds-\int_t^T Z_s^ndB_s+(A_T^{n}-A_t^{n}).
%\end{split}
	\end{equation}
	The objective is to show that $(Y^n,Z^n,A^n)$ converges to $(Y,Z,A)$ and $(Y,Z,A)$ is the solution of reflected $G$-BSDE with parameters $(\xi,f,S)$. The proof will be divided into the following steps.
%In the following of this section, $C$ will always denote a constant relying on $\alpha, T,\kappa, \underline{\sigma}$, but not on $n$.

\textbf{Step 1.} We first show that $Y^n$, $Z^n$, $K^n$, $L^{n}$ are uniformly bounded under the norm $\|\cdot\|_{S_G^\alpha}$, $\|\cdot\|_{H_G^\alpha}$, $\|\cdot\|_{L_G^\alpha}$ and $\|\cdot\|_{L_G^\alpha}$, respectively.   	
\begin{lemma}\label{Esti-Y}
For $2\leq \alpha<\beta$, there exists a constant $C$ independent of $n$, such that  \begin{align*}
&\hat{\mathbb{E}}[\sup_{t\in[0,T]}|Y_t^n|^\alpha]\leq C, \ \hat{\mathbb{E}}[|K_T^n|^\alpha]\leq C, \\
& \hat{\mathbb{E}}[|L_T^{n}|^\alpha]\leq C, \  \mathbb{\hat{E}}[(\int_{0}^{T}|Z^n_{s}|^{2}ds)^{\frac{\alpha}{2}}]\leq C.
\end{align*}	
\end{lemma}
	
\begin{proof} Recall that $I_t=I_0+\int_0^t b^I(s)ds+\int_0^t \sigma^I(s)dB_s+K^I_t$ is the generalized $G$-It\^o process such that $S_t\le I_t$, $t\in[0,T]$. 	Set
$\bar{Y}_t^n=Y_t^n-{I}_t,$  $\bar{Z}_t^n=Z_t^n-\sigma^I(t)$, $H_t=(\bar{Y}^n_t)^2$, $\bar{S}_t=S_t-I_t$, and $\bar{f}_t=f(t,Y^n_t, Z^n_t)+b^I(t)$. $G$-BSDE \eqref{equ1} can be rewritten as
\begin{align*}
\bar{Y}_t^n=&\xi-I_T+\int_t^T \bar{f}(s)ds+n\int_t^T(\bar{Y}_s^n-\bar{S}_s)^-ds -\int_t^T \bar{Z}_s^ndB_s\\
 &-(K_T^n-K_t^n) +(K^I_T-K^I_t).
\end{align*}
For any $r>0$, applying $G$-It\^{o}'s formula to $H_t^{\alpha/2}e^{rt}$, we get
\begin{displaymath}
\begin{split}
&H_t^{\alpha/2}e^{rt}+\int_t^T re^{rs}H_s^{\alpha/2}ds+\int_t^T \frac{\alpha}{2} e^{rs}H_s^{\alpha/2-1}(\bar{Z}_s^n)^2d\langle B\rangle_s\\
=&|\xi-{I}_T|^\alpha e^{rT}+\alpha(1-\frac{\alpha}{2})\int_t^Te^{rs}H_s^{\alpha/2-2}(\bar{Y}_s^n)^2(\bar{Z}_s^n)^2d\langle B\rangle_s+\int_t^T{\alpha} e^{rs}H_s^{\alpha/2-1}\bar{Y}_s^n\bar{f}_sds\\
&+\int_t^T\alpha e^{rs}H_s^{\alpha/2-1}n\bar{Y}_s^n(\bar{Y}_s^n-\bar{S}_s)^-ds-\int_t^T\alpha e^{rs}H_s^{\alpha/2-1}(\bar{Y}_s^n\bar{Z}_s^ndB_s+\bar{Y}_s^ndK_s^n-\bar{Y}^n_sdK^I_s).
\end{split}
\end{displaymath}
It is easy to check that  $\bar{Y}_s^n(\bar{Y}_s^n-\bar{S}_s)^-\le0$. Hence,	
\begin{displaymath}
\begin{split}
&\quad H_t^{\alpha/2}e^{rt}+\int_t^T re^{rs}H_s^{\alpha/2}ds+\int_t^T \frac{\alpha}{2} e^{rs}H_s^{\alpha/2-1}(\bar{Z}_s^n)^2d\langle B\rangle_s\\
&\leq|\xi-{I}_T|^\alpha e^{rT}+\alpha(1-\frac{\alpha}{2})\int_t^Te^{rs}H_s^{\alpha/2-2}(\bar{Y}_s^n)^2(\bar{Z}_s^n)^2d\langle B\rangle_s\\
&\quad+\int_t^T{\alpha} e^{rs}H_s^{\alpha/2-1/2}|\bar{f}_s|ds-(M_T-M_t),
\end{split}
\end{displaymath}	
where \[M_t=\int_0^t\alpha e^{rs}H_s^{\alpha/2-1}(\bar{Y}_s^n\bar{Z}_sdB_s+(\bar{Y}_s^n)^+dK_s^n+(\bar{Y}^n_s)^-dK^I_s)\] is a $G$-martingale.
Recalling Remark \ref{rho}, we have
\begin{align*}
&\int_t^T{\alpha} e^{rs}H_s^{\alpha/2-1/2}|\bar{f}_s|ds\\
\leq &\int_t^T{\alpha} e^{rs}H_s^{\alpha/2-1/2}\{|f(s,0,0)|+|b^I(s)|+\rho(|\bar{Y}_s^n|+|I_s|)+L[|\bar{Z}_s^n|+|\sigma^I(s)|]\}ds\\
\leq &(\alpha b+\frac{\alpha L^2}{\underline{\sigma}^2(\alpha-1)})\int_t^T e^{rs}H_s^{\alpha/2}ds+\frac{\alpha(\alpha-1)}{4}\int_t^Te^{rs}H_s^{\alpha/2-1}(\bar{Z}_s^n)^2d\langle B\rangle_s\\
&+\int_t^T \alpha e^{rs}H_s^{\alpha/2-1/2}[|f(s,0,0)|+|b^I(s)|+L|\sigma^I(s)|+a+b|I_s|] ds.
\end{align*}
By Young's inequality, we obtain
\begin{align*}
&\int_t^T \alpha e^{rs}H_s^{\alpha/2-1/2}[|f(s,0,0)|+|b^I(s)|+L|\sigma^I(s)|+a+b|I_s|] ds\\
\leq &\int_t^T e^{rs}[|f(s,0,0)|^\alpha+|b^I(s)|^\alpha+b^\alpha|{I}_s|^\alpha+L^\alpha|\sigma^I(s)|^\alpha+a^\alpha]ds\\
&+5(\alpha-1)\int_t^T e^{rs}H_s^{\alpha/2}ds .
\end{align*}
Combining the above inequalities, we get
\begin{align*}
&H_t^{\frac{\alpha}{2}}e^{rt}+\int_t^T (r-\tilde{\alpha})e^{rs}H_s^{\frac{\alpha}{2}}ds+\int_t^T \frac{\alpha(\alpha-1)}{4} e^{rs}H_s^{\frac{\alpha-2}{2}}(\bar{Z}_s^n)^2d\langle B\rangle_s+(M_T-M_t)\\
\leq &|\xi-{I}_T|^\alpha e^{rT}+\int_t^T e^{rs}[|f(s,0,0)|^\alpha+|b^I(s)|^\alpha+b^\alpha|{I}_s|^\alpha+L^\alpha|\sigma^I(s)|^\alpha+a^\alpha]ds,
\end{align*}
where $\tilde{\alpha}=5(\alpha-1)+\alpha b+\frac{\alpha L^2}{\underline{\sigma}^2(\alpha-1)}$. Setting $r=\tilde{\alpha}+1$ and taking conditional expectations on both sides, we derive that
\begin{displaymath}
H_t^{\alpha/2}e^{rt}\leq \hat{\mathbb{E}}_t[|\xi-{I}_T|^\alpha e^{rT}+\int_t^T e^{rs}[|f(s,0,0)|^\alpha+|b^I(s)|^\alpha+b^\alpha|{I}_s|^\alpha+L^\alpha|\sigma^I(s)|^\alpha+a^\alpha]ds].
\end{displaymath}
Then, there exists a constant $C$ independent of $n$ such that
\begin{displaymath}
|\bar{Y}_t^n|^\alpha\leq C\hat{\mathbb{E}}_t[|\xi-{I}_T|^\alpha+\int_t^T [1+|f(s,0,0)|^\alpha+|b^I(s)|^\alpha+|\sigma^I(s)|^\alpha+|{I}_s|^\alpha]ds].
\end{displaymath}
Noting that $|Y_t^n|^\alpha\leq C(|\bar{Y}_t^n|^\alpha+|I_t|^\alpha)$ and applying Theorem \ref{the1.2}, we  get that 
$$\hat{\mathbb{E}}[\sup_{t\in[0,T]}|Y^n_t|^\alpha]\leq C$$
 with $C$ independent of $n$. Since $0\leq -K^n_T\leq A^n_T$ and $0\leq L^n_T\leq A^n_T$, applying Proposition \ref{the1.6'} yields the uniform estimates for $Z^n$, $K^n$, $L^{n}$, respectively.
\end{proof}

\textbf{Step 2.} Now, we prove that $(Y^n-S)^-$ converges to $0$ uniformly.

\begin{lemma}\label{Conv-U-L}
Assume that (H1) and (A2)-(A4) hold.
For any $2\leq \alpha<\beta$, we have
\begin {eqnarray}\label {Conv.-U-L}
 \lim_{n\rightarrow \infty} \hat{\mathbb{E}}[\sup_{t\in[0,T]}|(Y_t^n-S_t)^-|^\alpha]=0.
\end {eqnarray}
\end{lemma}

\begin {proof}
 For each given $\varepsilon>0$, we can choose a Lipschitz function $l(\cdot)$ such that $I_{[-\varepsilon,\varepsilon]}\leq l(x)\leq I_{[-2\varepsilon,2\varepsilon]}$. Thus we have
\begin{align*}
&f(s,Y_s^n,Z_s^n)-f(s,Y_s^n,0)\\=&(f(s,Y_s^n, Z_s^n)-f(s, Y_s^n,0))l(Z_s^n)+a^{\varepsilon,n}_sZ_s^n=:m_s^{\varepsilon,n}+a^{\varepsilon,n}_sZ_s^n,
\end{align*}
where $a^{\varepsilon,n}_s=(1-l(Z_s^n))(f(s,Y_s^n, Z_s^n)-f(s, Y_s^n,0))(Z_s^n)^{-1}\in M_G^2(0,T)$ with $|a^{\varepsilon,n}_s|\leq L$. It is easy to check that $|m_s^{\varepsilon,n}|\leq 2L\varepsilon$. Then we can get
\begin{displaymath}
f(s, Y_s^n, Z_s^n)=f(s, Y_s^n,0)+a^{\varepsilon,n}_s Z_s^n+m_s^{\varepsilon,n}.
\end{displaymath}
Now we consider the following $G$-BSDE:
\begin{displaymath}
Y^{\varepsilon,n}_t=\xi+\int_t^T a^{\varepsilon,n}_sZ^{\varepsilon,n}_sds-\int_t^T Z^{\varepsilon,n}_sdB_s-(K^{\varepsilon,n}_T-K^{\varepsilon,n}_t).
\end{displaymath}
For each $\xi\in L_G^p(\Omega_T)$ with $p>1$, define
\begin{displaymath}
\tilde{\mathbb{E}}^{\varepsilon,n}_t[\xi]:=Y^{\varepsilon,n}_t,
\end{displaymath} which is a time-consistent sublinear expectation.
Set $\tilde{B}^{\varepsilon,n}_t=B_t-\int_0^t a^{\varepsilon,n}_s ds$. By Theorem 5.2 in \cite{HJPS2}, $\{\tilde{B}^{\varepsilon,n}_t\}$ is a $G$-Brownian motion under $\tilde{\mathbb{E}}^{\varepsilon,n}[\cdot]$.

We rewrite $G$-BSDE \eqref{equ1} as the following
\begin{eqnarray} \label {BSDE-Gisanov}
%\begin {split}
Y_t^n=\xi+\int_t^T f^{\varepsilon,n}(s)ds+\int_t^Tn(Y_s^n-S_s)^-ds
 -\int_t^T Z_s^nd\tilde{B}^{\varepsilon,n}_s-(K_T^n-K_t^n),
%\end {split}
\end{eqnarray}
where $f^{\varepsilon,n}(s)=f(s, Y_s^n,0)+m^{\varepsilon,n}_s$. Since $K^n$ is a martingale under $\tilde{\mathbb{E}}^{\varepsilon,n}[\cdot]$ by Theorem 5.1 in \cite {HJPS2},   it follows by Lemma 4.3 in \cite{LS}	that
\begin{align*}
(Y_t^n-S_t)^-\leq& \bigg|\tilde{\mathbb{E}}^{\varepsilon, n}_t[\int_t^Te^{-n(s-t)}f^{\varepsilon, n}(s)ds+\int_t^Te^{-n(s-t)}dS_s]\bigg|.
\end{align*}
By Proposition 3.7 in \cite{HJPS1} (i.e., a priori estimates for $G$-BSDEs), for $2\leq \alpha<\beta$, it follows that
		\begin{equation}\begin{split} \label {Conv.-U-Proof}
		&\hat{\mathbb{E}}[\sup_{t\in[0,T]}|(Y^n_t-S_t)^-|^\alpha]\\
\leq &\hat{\mathbb{E}}\bigg[\sup_{t\in[0,T]}\bigg|\tilde{\mathbb{E}}^{\varepsilon, n}_t[\int_t^Te^{-n(s-t)}f^{\varepsilon, n}(s)ds+\int_t^Te^{-n(s-t)}dS_s]\bigg|^{\alpha}\bigg]\\
		\leq & C_{\alpha}\hat{\mathbb{E}}\bigg[\sup_{t\in[0,T]}\hat{\mathbb{E}}_t[\bigg|\int_t^Te^{-n(s-t)}f^{\varepsilon, n}(s)ds+\int_t^Te^{-n(s-t)}dS_s\bigg|^{\alpha}]\bigg].
		\end{split}\end{equation}
Set 
$$I=\hat{\mathbb{E}}\bigg[\sup_{t\in[0,T]}\hat{\mathbb{E}}_t[\bigg|\int_t^Te^{-n(s-t)}dS_s\bigg|^{\alpha}]\bigg],$$
which converges to $0$ as $n$ goes to $\infty$ by Lemma 4.2 in \cite{LS}. Recalling Remark \ref{rho} and $|m_s^{\varepsilon,n}|\leq 2L\varepsilon$, we obtain that 
\begin{align*}
\int_t^Te^{-n(s-t)}f^{\varepsilon, n}(s)ds\leq &(\int_t^T e^{-2n(s-t)}ds)^{1/2}(\int_t^T |f^{\varepsilon,n}(s)|^2 ds)^{1/2}\\
\leq &\frac{C}{\sqrt{n}}(\int_0^T (1+|f(s,0,0)|^2+|Y_s^n|^2)ds)^{1/2}.
\end{align*}
Set
$$II=\hat{\mathbb{E}}\bigg[\sup_{t\in[0,T]}\hat{\mathbb{E}}_t[\bigg|\int_t^Te^{-n(s-t)}f^{\varepsilon, n}(s)ds\bigg|^\alpha]\bigg],$$
which converges to $0$ as $n$ goes to $\infty$ by Lemma \ref{Esti-Y}.
 The proof is complete. 
\end{proof}

\textbf{Step 3.} Now we prove the convergence of $Y^n$, $Z^n$, $A^n$.
	\begin{lemma}\label{Conv-Y-Z-A}
		For  $2\leq \alpha<\beta$, we have
		\begin{align*}&\lim_{n,m\rightarrow\infty}\hat{\mathbb{E}}[\sup_{t\in[0,T]}|Y_t^n-Y_t^m|^\alpha]=0,\\ &\lim_{n,m\rightarrow\infty}\hat{\mathbb{E}}[(\int_0^T|Z_s^n-Z_s^m|^2ds)^{\frac{\alpha}{2}}]=0, \\
		 &\lim_{n,m\rightarrow\infty}\hat{\mathbb{E}}[\sup_{t\in[0,T]}|A_t^n-A_t^m|^\alpha]=0.\end{align*}
	\end{lemma}
	
	\begin{proof}
		For any $r>0$,  and $n, m\in\mathbb{N}$, set
$$
\begin{array}{lll}{\hat{Y}_t=Y_t^n-Y_t^m,} & {\hat{Z}_t=Z_t^n-Z_t^m,} & {\hat{f}_t=f(t,Y_t^n,Z_t^n)-f(t,Y_t^m,Z_t^m),}\\ {\hat{L}_t=L_t^{n}-L_t^{m},}  & {\hat{K}_t=K_t^n-K_t^m.}  \end{array}
$$		
Denote $H_t=|\hat{Y}_t|^2$. Applying It\^{o}'s formula to $H_t^{\alpha/2}e^{rt}$, we  get
		\begin{displaymath}
		\begin{split}
		&\quad H_t^{\alpha/2}e^{rt}+\int_t^T re^{rs}H_s^{\alpha/2}ds+\int_t^T \frac{\alpha}{2} e^{rs}
		H_s^{\alpha/2-1}(\hat{Z}_s)^2d\langle B\rangle_s\\
		&=
		\alpha(1-\frac{\alpha}{2})\int_t^Te^{rs}H_s^{\alpha/2-2}(\hat{Y}_s)^2(\hat{Z}_s)^2d\langle B\rangle_s
		+\int_t^T\alpha e^{rs}H_s^{\alpha/2-1}\hat{Y}_sd\hat{L}_s\\
		&\quad+\int_t^T{\alpha} e^{rs}H_s^{\alpha/2-1}\hat{Y}_s\hat{f}_sds-\int_t^T\alpha e^{rs}H_s^{\alpha/2-1}(\hat{Y}_s\hat{Z}_sdB_s+\hat{Y}_sd\hat{K}_s).
		\end{split}
		\end{displaymath}
Recalling that  $L_t^{n}=n\int_0^t (Y_s^n-S_s)^-ds$, we have
\begin{align*}
 &\int_t^T\alpha e^{rs}H_s^{\alpha/2-1}\hat{Y}_sd(\hat{L}_s)\\
=&\int_t^T\alpha e^{rs}H_s^{\alpha/2-1}\bigg[(Y^n_s-S_s)-(Y^m_s-S_s)\bigg](dL^{n}_s-dL^{m}_s)\\
\le& \int_t^T\alpha e^{rs}H_s^{\alpha/2-1}\bigg[(Y^n_s-S_s)^-dL^{m}_s+(Y^m_s-S_s)^{-}dL^{n}_s\bigg]
=:\int_t^T\Delta^{n,m}_sds.
\end{align*}
Therefore,
\begin{displaymath}
		\begin{split}
		&\quad H_t^{\alpha/2}e^{rt}+\int_t^T re^{rs}H_s^{\alpha/2}ds+\int_t^T \frac{\alpha}{2} e^{rs}
		H_s^{\alpha/2-1}(\hat{Z}_s)^2d\langle B\rangle_s\\
		&\le\alpha(1-\frac{\alpha}{2})\int_t^Te^{rs}H_s^{\alpha/2-2}(\hat{Y}_s)^2(\hat{Z}_s)^2d\langle B\rangle_s+\int_t^T{\alpha} e^{rs}H_s^{\alpha/2-1}\hat{Y}_s\hat{f}_sds\\
		&\quad+\int_t^T\Delta^{n,m}_sds-(M_T-M_t),
		\end{split}
		\end{displaymath}
		where $M_t=\int_0^t \alpha e^{rs}H_s^{\alpha/2-1}(\hat{Y}_s\hat{Z}_sdB_s+(\hat{Y}_s)^+dK_s^m+(\hat{Y}_s)^-dK_s^n)$ is a $G$-martingale. By the assumption on $f$, applying the H\"{o}lder inequality and the Young inequality, we have
		\begin{align*}
		\int_t^T{\alpha} e^{rs}H_s^{\frac{\alpha-1}{2}}|\hat{f}_s|ds
		\leq &(\alpha -1+\frac{\alpha L^2}{\underline{\sigma}^2(\alpha-1)})\int_t^T e^{rs}H_s^{\alpha/2}ds+\int_t^T e^{rs} \rho^\alpha(|\hat{Y}_s|)ds\\
&+\frac{\alpha(\alpha-1)}{4}\int_t^Te^{rs}H_s^{\alpha/2-1}(\hat{Z}_s)^2d\langle B\rangle_s.
		\end{align*}
		Letting $r=\alpha+\frac{\alpha L^2}{\underline{\sigma}^2(\alpha-1)}$, we have
		\begin{align*}
		H_t^{\alpha/2}e^{rt}+(M_T-M_t)\leq	\int_t^T\Delta^{n,m}_sds+\int_t^T e^{rs} \rho^\alpha(|\hat{Y}_s|)ds.		
		\end{align*}
		Taking conditional expectation on both sides of the above inequality, it follows that
		\begin{equation}\label{eq1.5}
		H_t^{\alpha/2}e^{rt}\leq
		\hat{\mathbb{E}}_t[\int_t^T \Delta^{n,m}_sds+\int_t^T e^{rs} \rho^\alpha(|\hat{Y}_s|)ds].
		\end{equation}
		Taking expectation on both sides yields that 
		\begin{equation}\begin{split}\label{eq8}
		\hat{\mathbb{E}}[|\hat{Y}_t|^\alpha]\leq &
		C\{\hat{\mathbb{E}}[\int_0^T \Delta^{n,m}_sds]+\hat{\mathbb{E}}[\int_t^T \widetilde{\rho}(|\hat{Y}_s|^\alpha)ds]\}\\
		\leq &C\hat{\mathbb{E}}[\int_0^T \Delta^{n,m}_sds]+C\int_t^T \widetilde{\rho}(\hat{\mathbb{E}}[|\hat{Y}_s|^\alpha])ds,
		\end{split}\end{equation}
		where $\widetilde{\rho}(x):=\rho^\alpha(x^{1/\alpha})$ is a concave function by Lemma \ref{lem4} and we have used Lemma \ref{lem2.14} in the last inequality. By a similar analysis as the proof of Lemma 4.7 in \cite{LP}, for $1<\gamma<\beta/\alpha$, we have 
		\begin{equation}\label{Delta}
		\lim_{n,m\rightarrow\infty}\hat{\mathbb{E}}[(\int_0^T \Delta^{n,m}_s ds)^\gamma]=0.
		\end{equation}
		Set $u_{n,m}(t)=\sup_{r\in[t,T]}\hat{\mathbb{E}}[|\hat{Y}_r|^\alpha]$. It is easy to check that $u_{n,m}$ is uniformly bounded by Lemma \ref{Esti-Y}. It follows from \eqref{eq8} that 
		\begin{align*}
		u_{n,m}(t)\leq C\hat{\mathbb{E}}[\int_0^T \Delta^{n,m}_sds]+C\int_t^T \widetilde{\rho}(u_{n,m}(s))ds.
		\end{align*}
		Setting $v(t)=\limsup_{n,m\rightarrow\infty}u_{n,m}(t)$ and applying the Lebesgue dominated convergence theorem yield
		\begin{align*}
		v(t)\leq C\int_t^T \widetilde{\rho}(v(s))ds.
		\end{align*}
		Therefore, by Lemma \ref{lem2.16}, we obtain that 
		\begin{align}\label{eqsupY}
		\lim_{n,m\rightarrow\infty}\sup_{t\in[0,T]}\hat{\mathbb{E}}[|Y_t^n-Y^m_t|^\alpha]=0.
		\end{align}
		Recalling \eqref{eq1.5}, it is easy to check that 
		\begin{align*}
		\hat{\mathbb{E}}[\sup_{t\in[0,T]}|\hat{Y}_t|^\alpha]\leq C\hat{\mathbb{E}}[\sup_{t\in[0,T]}\hat{\mathbb{E}}_t[\int_0^T \Delta^{n,m}_sds+\int_0^T \widetilde{\rho}(|\hat{Y}_s|^\alpha) ]].
		\end{align*}
		By \eqref{Delta} and \eqref{eqsupY}, we finally obtain that for any $2\leq \alpha<\beta$
		\begin{align*}
		\lim_{n,m\rightarrow\infty}\hat{\mathbb{E}}[\sup_{t\in[0,T]}|Y_t^n-Y^m_t|^\alpha]=0.
		\end{align*}
		
		 By a similar analysis as in the proof of Theorem \ref{main}: the uniqueness result, for any $2\leq \alpha<\beta$, we get that
		\begin{align*}
		&\hat{\mathbb{E}}[(\int_0^T |\hat{Z}_t|^2dt)^{\frac{\alpha}{2}}]\leq C\{\hat{\mathbb{E}}[\sup_{t\in[0,T]}|\hat{Y}_t|^\alpha]+(\hat{\mathbb{E}}[\sup_{t\in[0,T]}|\hat{Y}_t|^\alpha])^{1/2}\},\\
		&\hat{\mathbb{E}}[\sup_{t\in[0,T]}|\hat{A}_t|^\alpha]\leq C\{\hat{\mathbb{E}}[\sup_{t\in[0,T]}|\hat{Y}_t|^\alpha]+\tilde{\rho}(\sup_{t\in[0,T]}\hat{\mathbb{E}}[|\hat{Y}_t|^\alpha])+\hat{\mathbb{E}}[(\int_0^T |\hat{Z}_t|^2dt)^{\frac{\alpha}{2}}]\}.
		\end{align*}
The proof is complete. 
\end{proof}

\textbf{Step 4.} It remains to prove $(Y^n,Z^n,A^n)$ converges to the solution of reflected $G$-BSDE.

\begin{proof}[Proof of Theorem \ref{main}: the existence result]
By Lemma \ref{Conv-Y-Z-A}, there exist $(Y,Z,A)\in\mathcal{S}^\alpha_G(0,T)$, such that
\begin{align*}
\lim_{n\rightarrow\infty}\hat{\mathbb{E}}[\sup_{t\in[0,T]}|Y_t^n-Y_t|^\alpha]=0,\ \lim_{n\rightarrow\infty}\hat{\mathbb{E}}[(\int_0^T|Z_s^n-Z_s|^2ds)^{\frac{\alpha}{2}}]=0, \
		 \lim_{n,\rightarrow\infty}\hat{\mathbb{E}}[\sup_{t\in[0,T]}|A_t^n-A_t|^\alpha]=0.
\end{align*}
It follows from Lemma \ref{Conv-U-L} that $S_t\leq Y_t$, $t\in[0,T]$. The proof for the fact that $\{-\int_0^t (Y_s-S_s)dA_s\}_{t\in[0,T]}$ is a nonincreasing $G$-martingale is the same as the proof of Theorem 5.1 in \cite{LPSH}. Therefore, $(Y,Z,A)$ is the solution to the reflected $G$-BSDE.
\end{proof}

As a byproduct of the penalization method, we have the following comparison theorem for reflected $G$-BSDEs, which extends the result Theorem 5.3 in \cite{LPSH}.

\begin{theorem}\label{comparison}
Let $(\xi^i,f^i,g^i,S^i)$ be two sets of parameters which satisfy (H1), (A2)-(A4), $i=1,2$. We additionally assume that 
\begin{itemize}
\item[(i)] $\xi^1\leq \xi^2$;
\item[(ii)] $f^1(t,y,z)\leq f^2(t,y,z)$, $g^1(t,y,z)\leq g^2(t,y,z)$, for any $(t,y,z)\in [0,T]\times\mathbb{R}\times\mathbb{R}$;
\item[(iii)] $S^1_t\leq S^2_t$, for any $t\in[0,T]$.
\end{itemize}
Let $(Y^i,Z^i,A^i)$ be the solutions of the reflected $G$-BSDE with parameters $(\xi^i,f^i,g^i,S^i)$, $i=1,2$, respectively. Then, we have $Y^1_t\leq Y^2_t$, $t\in[0,T]$.
\end{theorem}

\begin{proof}
Using the comparison theorem for $G$-BSDEs obtained in \cite{He} (see Theorem 3.7) and the construction for reflected $G$-BSDEs via penalization, we can easily obtain the desired result. So we omit it.
\end{proof}

%\section{}
%\label{AppendixB}

%\renewcommand{\theequation}{B-\arabic{equation}}

%%%%%%%%%%%%%%%%%%%%%%%%%%%%%%%%%%
	
%	\section{}
%\label{AppendixC}

%\renewcommand{\theequation}{C-\arabic{equation}}

%\section*{Acknowledgements}
%Li's research was supported by the National Nature Science Foundation of China and the Qilu Young Scholars Program of Shandong University.
%%%%%%%%%%%%%%%%%%%%%%%%%%%%%%%%%%%%%%%%%%%%%%%%%%%%%%%%%%%%%%%%%%%%%%%%%%%%%%%%%%%%%%%%%%%%%%%%%%%%%%%%%%%%%%%%%%%%%


\begin{thebibliography}{99}

\bibitem{BL} Bai, X. and Lin, Y. (2014) On the existence and uniqueness of solutions to stochastic differential equations driven by $G$-Brownian motion with integral-Lipschitz coefficients. Acta Math. Appl. Sin. Engl. Ser., 30(3): 589-610.

\bibitem{CT} Cao, D. and Tang, S. (2020) Reflected quadratic BSDEs driven by $G$-Brownian motions. Chinese Ann. Math. Ser. B, 41: 873-928.

\bibitem {DHP11} Denis L, Hu M, Peng S. (2011) Function spaces and
capacity related to a sublinear expectation: application to $G$-Brownian
motion pathes. Potential Anal. 34: 139-161.

\bibitem{FJ} Fan, S. and Jiang, L. (2019) $L^p$ solutions of BSDEs with a new kind of non-Lipschitz coefficients. Acta Math. Appl. Sin.   Engl. Ser., 35(4): 695-707.

\bibitem{He} He, W. (2022) BSDEs driven by $G$-Brownian motion with non-Lipschitz coefficients. J. Math. Anal. Appl., 

\bibitem{HJPS1} Hu, M., Ji, S., Peng, S. and Song Y.  (2014) {Backward stochastic differential equations
driven by $G$-Brownian motion}. Stochastic Processes and their Applications, 124: 759-784.

\bibitem{HJPS2}  Hu, M., Ji, S., Peng, S. and Song Y. (2014) {Comparison theorem, Feynman-Kac formula
and Girsanov transformation for BSDEs driven
by $G$-Brownian motion}. Stochastic Processes and their Applications, 124:  1170-1195.



\bibitem{HQW} Hu, M., Qu, B. and Wang, F. (2020) BSDEs driven by $G$-Brownian motion with time-varying Lipschitz condition. Journal of  Mathematical Analysis and Applications

\bibitem{HWZ} Hu, M., Wang, F. and Zheng, G. (2016) Quasi-continuous random variables and processes under the $G$-expectation framework. Stochastic Processes and their Applications, 126: 2367-2387.

\bibitem{HLS} Hu, Y., Lin, Y. and Soumana Hima, A. (2018) Quadratic backward stochastic differential equations driven by $G$-Brownian motion: discrete solutions and approximation. Stoch. Proc. Appl., 128(11): 3724-3750.

\bibitem{LP} Li, H. and Peng, S. (2020) Reflected backward stochastic differential equation driven by $G$-Brownian motion with an upper obstacle. Stochastic Processes and their Applications, 130: 6556-6579.

\bibitem{LPSH} Li, H., Peng, S. and Soumana Hima, A. (2018) Reflected solutions of backward stochastic differential equations driven by $G$-Brownian Motion. Sci. China Math., 61(1): 1–26.

\bibitem{LS} Li, H. and Song, Y. (2021) Backward stochastic differential equations driven by $G$-Brownian motion with double reflections. Journal of Theoretical Probability, 34: 2285-2314.

\bibitem{M} Mao, X. (1995) Adapted solutions of  backward stochastic differential equations with non-Lipschitz coefficients. Stoch. Process. Appl., 58(2): 281-292. 

\bibitem {P07a} Peng S. (2007) $G$-expectation, $G$-Brownian Motion and
Related Stochastic Calculus of It\^o type. Stochastic analysis and
applications, Abel Symp., 2, pp.541-567. Springer, Berlin. 

\bibitem {P08a} Peng S. (2008) Multi-dimensional $G$-Brownian motion and
related stochastic calculus under $G$-expectation. Stochastic Processes and
their Applications. 118(12): 2223-2253.

\bibitem {P10} Peng S.: Nonlinear expectations and stochastic
calculus under uncertainty. arXiv:1002.4546v1, (2010)


\bibitem{S11} Song, Y. (2011) Some properties on $G$-evaluation and its applications to $G$-martingale decomposition. Sci. China Math., 54: 287-300.

\bibitem{S} Sun, S. (2022) Backward stochastic differential equations driven by $G$-Brownian motion with uniformly continuous coefficients in (y,z). J. Theor. Probab., 35, 370-409.

\bibitem{WZ} Wang, F. and Zheng, G. (2021) Backward stochastic differential equations driven by $G$-Brownian motion with uniformly continuous generators. J. Theor. Probab., 34, 660-681.



\end{thebibliography}
\end{document}